\RequirePackage{rotating} 
\documentclass[opre,nonblindrev]{informs3} 

\OneAndAHalfSpacedXI 

\usepackage{endnotes}
\let\footnote=\endnote

\usepackage{natbib}
 \bibpunct[, ]{(}{)}{,}{a}{}{,}%
 \def\bibfont{\footnotesize}%
 \def\BIBand{and}%

\usepackage[utf8]{inputenc}
\usepackage[english]{babel}
\usepackage[margin=1in]{geometry}
\usepackage{bbding}
\usepackage[shortlabels]{enumitem}
\usepackage[super]{nth}
\usepackage{bm, bbm}
\usepackage{amsmath, amssymb}
\usepackage{mathtools}
\usepackage{subcaption, placeins}
\usepackage{graphicx}
\usepackage[hidelinks]{hyperref}
\usepackage{pgfplots}

\usepackage{algorithm}
\usepackage{algorithmic}
\usepackage{subcaption}
\usepackage{booktabs}
\usepackage{multirow} 
\usepackage{placeins}
\usepackage{booktabs}
\usepackage{accents}
\usepackage{graphics}

\DeclarePairedDelimiter\floor{\lfloor}{\rfloor}
\usepackage{tikz}
\usepackage{placeins}
\usepackage{float}
\usepackage{graphicx}
\usepackage{accents}
\usepackage{rotating}
\usepackage{graphics}

\usepackage{color}
\newcommand{\blue}{\color{black}}

\definecolor{blue}{rgb}{0,0,1}

\setlist[itemize]{leftmargin=*, nosep}

\usepackage{etoolbox}
\newcommand{\zerodisplayskips}{%
  \setlength{\abovedisplayskip}{7pt}%
  \setlength{\belowdisplayskip}{5pt}%
  \setlength{\abovedisplayshortskip}{0pt}%
  \setlength{\belowdisplayshortskip}{0pt}}
\appto{\normalsize}{\zerodisplayskips}
\appto{\small}{\zerodisplayskips}
\appto{\footnotesize}{\zerodisplayskips}
\usepackage{xfrac}

\usepackage[sectionbib]{bibunits}
    \defaultbibliographystyle{informs2014} 
    \defaultbibliography{thebib}
    
\TheoremsNumberedThrough     
\ECRepeatTheorems

\EquationsNumberedThrough    

\begin{document}


 \RUNAUTHOR{Cory-Wright and Pauphilet}

\RUNTITLE{Compact Lifted Relaxations for Low-Rank Optimization}

\TITLE{Compact Lifted Relaxations for Low-Rank Optimization}

\ARTICLEAUTHORS{%
\AUTHOR{Ryan Cory-Wright}
\AFF{Department of Analytics, Marketing and Operations, Imperial Business School, London, UK\\
ORCID: \href{https://orcid.org/0000-0002-4485-0619}{$0000$-$0002$-$4485$-$0619$}\\ \EMAIL{r.cory-wright@imperial.ac.uk}} 
\AUTHOR{Jean Pauphilet}
\AFF{Management Science and Operations, London Business School, London, UK \\ORCID: \href{https://orcid.org/0000-0001-6352-0984}{$0000$-$0001$-$6352$-$0984$}\\\EMAIL{jpauphilet@london.edu}}
} 

\ABSTRACT{
We develop tractable convex relaxations for rank-constrained quadratic optimization problems over $n \times m$ matrices, a setting for which tractable relaxations are typically only available when the objective or constraints admit spectral structure. We derive lifted semidefinite relaxations that do not require such spectral terms. Although a direct lifting introduces a large semidefinite constraint in dimension $n^2 + nm + 1$, we prove that many blocks of the moment matrix are redundant and derive an equivalent compact relaxation that only involves two semidefinite constraints of dimension $nm + 1$ and $n+m$, respectively. 
{\color{black}We also derive a new class of valid inequalities for low-rank problems, which we call projection cuts, that exploit the fact that rank constraints are inherited by linear images of a low-rank matrix, to strengthen our low-rank relaxations substantially.} For matrix completion and reduced-rank regression problems, among others, we exploit additional structure to obtain even more compact formulations involving semidefinite matrices of dimension at most {\color{black}the sum of the two dimensions of the low-rank decision matrix (i.e., of size at most $n+m$)}. Overall, we obtain scalable semidefinite bounds for a broad class of low-rank quadratic problems. 
}%

\KEYWORDS{Low-rank optimization; semidefinite programming; matrix completion} 

\maketitle
\vspace{-10mm}

\begin{bibunit}

\section{Introduction}

This work develops tractable relaxations for rank-constrained optimization problems of the form
\begin{align}\label{prob:lowrankquadratic}
    \min_{\bm{X} \in \mathbb{R}^{n \times m}} \quad & \lambda \cdot \operatorname{rank}(\bm{X}) + \left\langle \bm{H}, \mathrm{vec}(\bm{X}^\top)\mathrm{vec}(\bm{X}^\top)^\top \right\rangle + \left\langle \bm{D}, \bm{X}\right\rangle\\
    \text{s.t.} \quad & \operatorname{rank}(\bm{X}) \leq k, \ \left\langle \bm{Q}_i, \mathrm{vec}(\bm{X}^\top)\mathrm{vec}(\bm{X}^\top)^\top\right\rangle+\langle \bm{E}_i, \bm{X}\rangle \leq b_i \ \forall i \in \mathcal{I} \nonumber
\end{align}
where $\bm{H}, \bm{Q}_i \in \mathcal{S}^{nm}$ are $nm \times nm$ symmetric matrices, $\bm{D}, \bm{E}_i$ are $n \times m$ matrices, $\lambda \in \mathbb{R}_+, k \in \mathbb{N}$ are parameters which control the complexity of $\bm{X} \in \mathbb{R}^{n \times m}$ by respectively penalizing and constraining its rank, $\mathcal{I}$ denotes the index set of constraints, and $\mathrm{vec}(\cdot)$ denotes the vectorization of the matrix $\bm{X}$. Note that we write $\mathrm{vec}(\bm{X}^\top)$ rather than $\mathrm{vec}(\bm{X})$ to simplify the notation in our {\color{black}subsequent} relaxations; both formulations are equivalent up to a fixed permutation of the coordinates. 

Problem \eqref{prob:lowrankquadratic} is a very general class of problems: it models unregularized and Frobenius-regularized matrix completion \citep{candes2009exact} and reduced rank regression \citep{negahban2011estimation} as special cases, as we discuss in detail in Section \ref{sec:examplesofrelaxations}. However, to our knowledge, there are no tractable convex relaxations for low-rank quadratic optimization problems like \eqref{prob:lowrankquadratic} except those that exploit a spectral term in the objective or constraints that may not persist in practice. Accordingly, this paper proposes computationally tractable lifted relaxations of Problem \eqref{prob:lowrankquadratic}.

The closest works to this paper are (i) \citet{kim2021convexification}, who generalize the work of \cite{hiriart2012convexifying} to propose an extended formulation for sets of the form $\{ \bm{X} \: : \: \operatorname{rank}(\bm{X}) \leq k,\ f(\bm{X}) \leq 0 \}$, where $f$ is 
a quasiconvex function in the singular values of $\bm{X}$, 
(ii) \citet{bertsimas2021new}, who{\color{black}, building upon the work of \cite{bertsimas2020mixed}, }leverage partial separability to obtain compact semidefinite relaxations for low-rank problems with a spectral term in the objective of the form $\mathrm{tr}(f(\bm{X}))$ for $f$ matrix convex, via a matrix perspective relaxation, and (iii) \cite{li2025partial}, who propose a column generation algorithm for solving the convex relaxation of a low-rank spectrally constrained problem. All three works yield powerful relaxations when their structural assumptions hold, but do not yield generic, scalable convex relaxations for general low-rank quadratic optimization in which the objective and constraints are
arbitrary quadratic forms in $\mathrm{vec}(\bm X^\top)$. We complement this literature by providing a lifted SDP relaxation for general quadratic objectives and constraints, and by showing how to systematically eliminate variables to obtain implementable relaxations for low-rank problems.


\textbf{Main contributions:} In this work, we extend the lifted relaxation idea of \cite{shor1987quadratic} to Problem~\eqref{prob:lowrankquadratic} by leveraging the 
{\color{black}representation of low-rank matrices with projection matrices as proposed in}
\citet{bertsimas2020mixed}, thus enriching the toolbox of semidefinite relaxations for low-rank optimization. Our main theoretical contributions {\color{black} are a new class of compact} semidefinite relaxations for generic low-rank optimization problems (Theorem~\ref{thm:shorlowrankequiv} in Section \ref{ssec:shorrelax}), {\color{black}and a new family of valid inequalities, which we term projection cuts (Theorem \ref{prop:proj_cuts} in Section \ref{ssec:strongershor}), that often close most of the gap between this convex relaxation and feasible solutions}. To the best of our knowledge, this is the first work that obtains non-trivial and computationally tractable lower bounds for {\color{black} generic} quadratic low-rank optimization problems like \eqref{prob:lowrankquadratic}, without depending on spectral terms. 

\textbf{Structure:} We propose lifted relaxations of low-rank quadratic optimization problems in Section~\ref{sec:lowrankshorrelax}. To illustrate our approach, we apply our lifted relaxation to {\color{black}low-rank matrix completion and reduced rank regression in Section \ref{sec:examplesofrelaxations}. In Section \ref{sec:numerics}, we numerically benchmark our convex relaxations on low-rank matrix completion problems.}

\subsection{Notation} \label{ssec:notation}
We let non-boldface characters such as $b$ denote scalars, lowercase boldface characters such as $\bm{x}$ denote vectors, uppercase boldface characters such as $\bm{X}$ denote matrices, and calligraphic uppercase characters such as $\mathcal{Z}$ denote sets. We let $[n]$ denote the set of running indices $\{1, ..., n\}$. The cone of $n \times n$ symmetric (resp. positive {\color{black}semi}definite) matrices is denoted by $\mathcal{S}^n$ (resp. $\mathcal{S}^n_+$). 


For a matrix $\bm{X} \in \mathbb{R}^{n \times m}$, we let 
$\bm{X}_i$ denote its $i$th column and $\bm{X}_{i,.}$ denote a vector containing its $i$th row. We let $\mathrm{vec}(\bm{X}): \mathbb{R}^{n \times m} \rightarrow \mathbb{R}^{nm}$ denote the vectorization operator which maps matrices to vectors by stacking columns. 
For a matrix $\bm{W}$, we may find it convenient to describe it as a block matrix composed of equally sized blocks and denote the $(i,i')$ block by $\bm{W}^{(i,i')}$. The dimension of each block will be clear from the context, given the size of the matrix $\bm{W}$ and the number of blocks. 
{\color{black}For any matrices $\bm{A},\bm{B}$, we use $\bm{A} \otimes \bm{B}$ to denote their Kronecker product \citep[see][Chapter 1.2]{gupta2018matrix}.} 

We let $\bm{X}^\dag$ be the pseudoinverse of $\bm{X}$, which {\color{black}is used} in the Schur complement lemma \citep[][Eqn. 2.41]{boyd1994linear}. We let $\mathcal{Y}_n^k:=\{\bm{Y} \in \mathcal{S}^n_+ : \bm{Y}^2 =\bm{Y}, \mathrm{tr}(\bm{Y}) \leq k\}$ denote the set of orthogonal projection matrices with rank at most $k$, whose convex hull is $\{\bm{P} \in \mathcal{S}^n_+: \bm{P} \preceq \bm{I}_n, \mathrm{tr}({\bm{P}}) \leq k\}$
\citep[][Theorem~3]{overton1992sum}. 

\section{Lifted Relaxations for Low-Rank Optimization Problems}\label{sec:lowrankshorrelax}
In this section, we demonstrate how to apply a lifted relaxation technique to low-rank quadratic optimization problems in a manner that yields tractable convex relaxations. Throughout the section, we study the following reformulation of Problem \eqref{prob:lowrankquadratic}:
\begin{align}\label{prob:lowrankrelaxation_orig}
    \min_{\bm{Y} \in \mathcal{Y}^k_n} \min_{\bm{X} \in \mathbb{R}^{n \times m}} \quad & \lambda \cdot \mathrm{tr}(\bm{Y})+\left\langle \mathrm{vec}(\bm{X}^\top)\mathrm{vec}(\bm{X}^\top)^\top, \bm{H}\right\rangle+\left\langle \bm{D}, \bm{X}\right\rangle\ \\
    \text{s.t.} \quad & \left\langle \bm{Q}_i, \mathrm{vec}(\bm{X}^\top)\mathrm{vec}(\bm{X}^\top)^\top\right\rangle+\langle \bm{E}_i, \bm{X}\rangle \leq b_i, \ \forall i \in {\color{black}\mathcal{I}}, \ \bm{X}=\bm{Y}\bm{X}.\nonumber
\end{align}

As demonstrated\footnote{The constraints $\bm{X}=\bm{YX}, \bm{Y}^2=\bm{Y}$ imply that $\mathrm{rank}(\bm{X})\leq \mathrm{tr}(\bm{Y})$, which can be made to hold with equality by letting $\bm{Y}=\bm{UU}^\top$ for $\bm{X}=\bm{U\Sigma V^\top}$ a singular value decomposition of $\bm{X}$.} in \citet{bertsimas2020mixed}, any rank-constrained optimization problem of the form \eqref{prob:lowrankquadratic} can be formulated as an optimization over $(\bm{X},\bm{Y})$ of the form \eqref{prob:lowrankrelaxation_orig}, where the additional decision matrix $\bm{Y}$ is a projection matrix which encodes the span of $\bm{X}$ and whose trace bounds $\operatorname{rank}(\bm{X})$. However, unlike Problem \eqref{prob:lowrankquadratic}, Problem \eqref{prob:lowrankrelaxation_orig} concentrates all {\color{black}rank-induced} nonconvexity within the constraint{\color{black}s} $\bm{Y}^2=\bm{Y}$ {\color{black}and $\bm{X}=\bm{YX}$}, making it amenable to lifted relaxation techniques.

The rest of the section is organized as follows: First, in Section \ref{ssec:shorrelax}, we design a lifted semidefinite relaxation {\color{black}for Problem \eqref{prob:lowrankrelaxation_orig} (Proposition \ref{prop:convexrelax}) and show that many of the additional semidefinite variables can be safely omitted (Theorem \ref{thm:shorlowrankequiv})}. 
{\color{black}Second, in Section \ref{ssec:strongershor}, we develop a family of additional valid inequalities which we call projection cuts. Third, in Section \ref{ssec:strongerresults_blockdiag}, we specialize our results to the case where $\bm{H}$ and $\bm{Q}_i$ are block diagonal matrices, and show that our semidefinite relaxations and projection cuts can be written in terms of optimization problems that do not involve any $n^2 \times n^2$ matrices in this case. Finally, in Section \ref{ssec:bmfailure}, we revisit the nonlinear formulation of \cite{burer2003nonlinear}, and discuss why a naive lifted relaxation designed around this nonlinear formulation cannot yield strong lower bounds, while a more sophisticated version of the relaxation is equivalent to the relaxation designed in Section \ref{ssec:shorrelax}. }

\subsection{A New Lifted Relaxation and Its Compact Formulation}\label{ssec:shorrelax}
We introduce matrices $\bm{W}_{x,x}$, $\bm{W}_{x,y}$, $\bm{W}_{y,y}$ to model the outer products $\mathrm{vec}(\bm{X}^\top)\mathrm{vec}(\bm{X}^\top)^\top$, $\mathrm{vec}(\bm{X}^\top)\mathrm{vec}(\bm{Y})^\top$, and $\mathrm{vec}(\bm{Y})\mathrm{vec}(\bm{Y})^\top$. By including these decision matrices, we have:
\begin{proposition}\label{prop:convexrelax} The convex semidefinite optimization problem 
\begin{equation}
\begin{aligned}\label{eqn:quadraticShorrelax_lowrank}
    \min_{\substack{ \bm{Y} \in \mathcal{S}^n_+ \, : \, \bm{Y} \preceq \bm{I},\, \mathrm{tr}(\bm{Y}) \leq k \\ \bm{W}_{y,y} \in \mathcal{S}^{n^2}_+}} \quad  \min_{\substack{\bm{X} \in  \mathbb{R}^{\blue n \times m}, \\ \bm{W}_{x,x} \in \mathcal{S}^{nm}_+, \bm{W}_{x,y} \in \mathbb{R}^{nm \times n^2}}} \quad & \lambda \cdot \mathrm{tr}(\bm{Y})+\left\langle \bm{W}_{x,x}, \bm{H} \right\rangle+\langle \bm{D}, \bm{X}\rangle \\
    \text{s.t.} \quad 
    & \begin{pmatrix} 1 & \mathrm{vec}(\bm{X}^\top)^\top & \mathrm{vec}(\bm{Y})^\top \\
    \mathrm{vec}(\bm{X}^\top) & \bm{W}_{x,x} & \bm{W}_{x,y}\\
    \mathrm{vec}(\bm{Y}) & \bm{W}_{x,y}^\top & \bm{W}_{y,y}\end{pmatrix} \succeq \bm{0},\\
    & \sum_{i=1}^n \bm{W}_{y,y}^{(i,i)} = \bm{Y},\ \sum_{i=1}^n \bm{W}_{x,y}^{(i,i)} = \bm{X}^\top, \\
    & \langle \bm{Q}_i, \bm{W}_{x,x}\rangle+\langle \bm{E}_i, \bm{X}\rangle\leq b_i \ \forall i \in \mathcal{I}, \\
\end{aligned}
\end{equation}
is a valid convex relaxation of Problem \eqref{prob:lowrankrelaxation_orig}.
\end{proposition}
\begin{remark}
    If an optimal solution to \eqref{eqn:quadraticShorrelax_lowrank} is such that $\bm{W}_{x,x}=\mathrm{vec}(\bm{X}^\top)\mathrm{vec}(\bm{X}^\top)^\top$ then the optimal values of \eqref{eqn:quadraticShorrelax_lowrank} and \eqref{prob:lowrankrelaxation_orig} coincide. 
\end{remark}

\proof{Proof of Proposition \ref{prop:convexrelax}}
Fix $(\bm{X}, \bm{Y})$ in \eqref{prob:lowrankrelaxation_orig} and set 
$$(\bm{W}_{x,x}, \bm{W}_{x,y}, \bm{W}_{y,y}):=(\mathrm{vec}(\bm{X}^\top) \mathrm{vec}(\bm{X}^\top)^\top, \mathrm{vec}(\bm{X}^\top) \mathrm{vec}(\bm{Y})^\top, \mathrm{vec}(\bm{Y}) \mathrm{vec}(\bm{Y})^\top).$$ 
It is sufficient to verify that $(\bm{X}, \bm{Y}, \bm{W}_{x,x}, \bm{W}_{x,y}, \bm{W}_{y,y})$ is feasible for \eqref{eqn:quadraticShorrelax_lowrank}---it obviously attains the same objective value. First, by construction, the semidefinite constraint is satisfied (at equality). Moreover, we have 
\begin{align*}
    \bm{YY}^\top = \bm{Y} &\implies \sum_{i=1}^n \bm{W}_{y,y}^{(i,i)} = \bm{Y}, \\
    \bm{X}^\top \bm{Y}^\top = \bm{X}^\top &\implies \sum_{i \in [n]}\bm{W}_{x,y}^{(i,i)}=\bm{X}^\top. 
\end{align*}
\hfill\Halmos
\endproof

Unfortunately, \eqref{eqn:quadraticShorrelax_lowrank} is not compact as it involves one semidefinite constraint of dimension $n^2 + nm + 1$. 
Therefore, a natural research question is whether it is possible to eliminate any variables from \eqref{eqn:quadraticShorrelax_lowrank} without altering its optimal objective value. We now answer this question affirmatively. 

\begin{theorem}\label{thm:shorlowrankequiv}
Problem \eqref{eqn:quadraticShorrelax_lowrank} is equivalent to
\begin{equation}
\begin{aligned}\label{eqn:quadraticShorrelax_lowrank2}
    \min_{\bm{Y} \in \mathcal{S}^n_+ \, : \, \bm{Y} \preceq \bm{I},\ \mathrm{tr}(\bm{Y}) \leq k} \quad 
    \min_{\substack{\bm{X} \in  \mathbb{R}^{n \times m} \, \\ \bm{W}_{x,x} \in \mathcal{S}^{nm}_+}} \quad & \lambda \cdot \mathrm{tr}(\bm{Y})+\left\langle \bm{W}_{x,x},  \bm{H} \right\rangle+\langle \bm{D}, \bm{X}\rangle \\
    \text{s.t.} \quad 
    & \bm{W}_{x,x}\succeq \mathrm{vec}(\bm{X}^\top)\mathrm{vec}(\bm{X}^\top)^\top,\\
    & \langle \bm{Q}_i, \bm{W}_{x,x}\rangle+\langle \bm{E}_i, \bm{X}\rangle\leq b_i, \ \forall i \in \mathcal{I},\\
    & \begin{pmatrix}
        \sum_{i \in [n]}\bm{W}_{x,x}^{(i,i)} & \bm{X}^\top\\ \bm{X} & \bm{Y}
    \end{pmatrix}\succeq \bm{0}.
    \end{aligned}
\end{equation} 
\end{theorem}

\proof{Proof of Theorem \ref{thm:shorlowrankequiv}}
We show that given a feasible solution to either problem, we can generate a feasible solution to the other problem with an equal or lower objective value.

Suppose that $(\bm{X}, \bm{Y}, \bm{W}_{x,x}, \bm{W}_{x,y}, \bm{W}_{y,y})$ is feasible in \eqref{eqn:quadraticShorrelax_lowrank}. Then, by summing the PSD principal submatrices containing $\bm{W}_{x,x}^{(i,i)}, \bm{W}_{x,y}^{(i,i)}, \bm{W}_{y,y}^{(i,i)}$ for each $i \in [n]$, we have that 
\begin{align*}
 \begin{pmatrix}
        \sum_{i \in [n]}\bm{W}_{x,x}^{(i,i)} & \sum_{i \in [n]}\bm{W}_{x,y}^{(i,i)}\\ \sum_{i \in [n]}\bm{W}_{x,y}^{(i,i)\top} & \sum_{i \in [n]}\bm{W}_{y,y}^{(i,i)}
    \end{pmatrix} \succeq \bm{0}.  
\end{align*}
Moreover, from \eqref{eqn:quadraticShorrelax_lowrank} we have that $\sum_{i \in [n]}\bm{W}_{\blue x,y}^{(i,i)}=\bm{X}^\top$ and $\sum_{i \in [n]}\bm{W}_{y,y}^{(i,i)}=\bm{Y}$. Thus, $(\bm{X}, \bm{Y}, \bm{W}_{x,x})$ is feasible in \eqref{eqn:quadraticShorrelax_lowrank2} and attains the same objective value.

Next, suppose that $(\bm{X}, \bm{Y}, \bm{W}_{x,x})$ is feasible in \eqref{eqn:quadraticShorrelax_lowrank2}. 
By the Schur complement lemma, we must have $\bm{Y}\succeq \bm{X} (\sum_{i}\bm{W}_{x,x}^{(i,i)})^\dagger \bm{X}^\top$. Since the cost matrix associated with $\bm{Y}$ in the objective, $\lambda \bm{I}$, is positive semidefinite, we can set $\bm{Y}:=\bm{X} (\sum_{i}\bm{W}_{x,x}^{(i,i)})^\dagger \bm{X}^\top$ without loss of optimality---doing so cannot increase the objective value nor make the constraints $\bm{0} \preceq \bm{Y} \preceq \bm{I}_n$ and $\operatorname{tr}(\bm{Y}) \leq k$ violated. To construct admissible matrices $\bm{W}_{x,y}$ and $\bm{W}_{y,y}$, let us first define the auxiliary matrix
\begin{align*}
\bm{U}:=\left(\sum_{i \in [n]}\bm{W}_{x,x}^{(i,i)}\right)^\dagger \bm{X}^\top \in \mathbb{R}^{m \times n},   
\end{align*}
and observe that $\bm{Y}=\bm{U}^\top \bm{X}^\top = \bm{X} \bm{U}$. Then, we define the matrix 
\begin{align*}
\bm{M} := 
& \begin{pmatrix} 1 & \bm{0} & \bm{0} \\ \bm{0} & \bm{I}_{nm} & \bm{0} \\ \bm{0} & \bm{0} & \bm{I}_{n} \otimes \bm{U} \end{pmatrix}^\top 
 \begin{pmatrix} 
   1 & \mathrm{vec}(\bm{X}^\top)^\top &  \mathrm{vec}(\bm{X}^\top)^\top \\ \mathrm{vec}(\bm{X}^\top) & \bm{W}_{x,x} & \bm{W}_{x,x}\\
   \mathrm{vec}(\bm{X}^\top) & \bm{W}_{x,x} & \bm{W}_{x,x}\end{pmatrix} \begin{pmatrix} 1 & \bm{0} & \bm{0} \\ \bm{0} & \bm{I}_{nm} & \bm{0} \\ \bm{0} & \bm{0} & \bm{I}_n \otimes \bm{U} \end{pmatrix}.
\end{align*}
We let $\bm{W}_{x,y}$, $\bm{W}_{y,y}$ denote the relevant off-diagonal blocks, i.e., 
\begin{align*}
\bm{M} =: \begin{pmatrix} 
1 & \mathrm{vec}(\bm{X}^\top)^\top & \mathrm{vec}(\bm{Y})^\top\\
\mathrm{vec}(\bm{X}^\top) & \bm{W}_{x,x} & \bm{W}_{x,y}\\
 \mathrm{vec}(\bm{Y})&    \bm{W}_{x,y}^\top & \bm{W}_{y,y}\end{pmatrix}.
\end{align*}
Since $\bm{Y}= \bm{U}^\top \bm{X}^\top$, we have $\mathrm{vec}(\bm{Y}) = \mathrm{vec}(\bm{U}^\top \bm{X}^\top) = (\bm{I}_n \otimes \bm{U}^\top) \mathrm{vec}(\bm{X}^\top)$ and thus our construction is consistent with the existing value of $\bm{Y}$. We now verify that $(\bm{X}, \bm{Y}, \bm{W}_{x,x}, \bm{W}_{x,y}, \bm{W}_{y,y})$ is feasible for \eqref{eqn:quadraticShorrelax_lowrank}.
By construction, $\bm{M} \succeq \bm{0}$. 
Thus, $(\bm{X}, \bm{Y}, \bm{W}_{x,x}, \bm{W}_{x,y}, \bm{W}_{y,y})$ satisfies the semidefinite constraint in \eqref{eqn:quadraticShorrelax_lowrank}.
Next, by construction, $\bm{W}_{x,y}$ and $\bm{W}_{y,y}$ can be decomposed into $n \times n$ blocks satisfying:
\begin{align*}
  \bm{W}_{x,y}^{(i,j)}= \bm{W}_{x,x}^{(i,j)}\bm{U}, \ \bm{W}_{y,y}^{(i,j)}=\bm{U}^\top \bm{W}_{x,x}^{(i,j)}\bm{U}.
\end{align*}
Summing the on-diagonal blocks of these matrices then reveals that
\begin{align*}
  \sum_{i \in [n]}\bm{W}_{x,y}^{(i,i)}=&\sum_{i \in [n]}\bm{W}_{x,x}^{(i,i)}\bm{U} = \left( \sum_{i \in [n]}\bm{W}_{x,x}^{(i,i)} \right) \left(\sum_{j \in [n]}\bm{W}_{x,x}^{(j,j)}\right)^\dagger \bm{X}^\top=\bm{X}^\top,\\
  \sum_{i \in [n]}\bm{W}_{y,y}^{(i,i)}=& \sum_{i \in [n]}\bm{U}^\top\bm{W}_{x,x}^{(i,i)}\bm{U} =
  \bm{U}^\top \left( \sum_{i \in [n]}\bm{W}_{x,x}^{(i,i)}\bm{U} \right) = \bm{U}^\top \bm{X}^\top = \bm{Y}.
\end{align*}
Therefore, we conclude that $(\bm{X}, \bm{Y}, \bm{W}_{x,x}, \bm{W}_{x,y}, \bm{W}_{y,y})$ is feasible in \eqref{eqn:quadraticShorrelax_lowrank} and attains an equal or lower objective value. Thus, both relaxations are equivalent.\hfill\Halmos
\endproof

Problem \eqref{eqn:quadraticShorrelax_lowrank2} is much more compact {\color{black}than} \eqref{eqn:quadraticShorrelax_lowrank}, as it does not require {\color{black}introducing} the variables $\bm{W}_{y,y} \in \mathcal{S}^{n^2}_+$ or $\bm{W}_{x,y} \in \mathbb{R}^{nm \times n^2}$. 
Instead, \eqref{eqn:quadraticShorrelax_lowrank2} only involves two semidefinite constraints of dimension $nm + 1$ and $n+m$ respectively.
Notably, the proof of Theorem \ref{thm:shorlowrankequiv} also provides a recipe for reconstructing an optimal $\bm{W}_{y,y}$ given an optimal solution ($\bm{Y}, \bm{X}, \bm{W}_{x,x}$) to \eqref{eqn:quadraticShorrelax_lowrank2}. Namely, compute the auxiliary matrix $\bm{U}:=\left(\sum_{i \in [n]}\bm{W}_{x,x}^{(i,i)}\right)^\dagger \bm{X}^\top$ and set $\bm{W}_{y,y}:=(\bm{I}_n \otimes \bm{U})^\top \bm{W}_{x,x}(\bm{I}_n \otimes \bm{U})$.




Finally, it is interesting to consider whether the relaxation developed here is at least as strong as the matrix perspective relaxation of \citet{bertsimas2021new}. We now prove this is indeed the case. 
{\color{black}The relaxation of }\cite{bertsimas2021new} only applies to partially separable objectives. Hence, we first need to impose more structure on the objective of \eqref{prob:lowrankrelaxation_orig} to compare relaxations. 

\begin{proposition}\label{prop:mprt_is_weaker}
Assume {\color{black}that $\bm{H}$ is positive definite}, so that the term $\langle \bm{H}, \mathrm{vec}(\bm{X}^\top)\mathrm{vec}(\bm{X}^\top)^\top\rangle+\langle \bm{D}, \bm{X}\rangle$ in Problem \eqref{prob:lowrankrelaxation_orig} can be rewritten as the partially separable term $\frac{1}{2\gamma}\Vert \bm{X}\Vert_F^2+h(\bm{X})$, where $h$ is convex in $\bm{X}$. Further, {\color{black}assume the matrices $\bm{Q}_i$ are positive definite, so that} the terms $\langle \bm{Q}_i, \mathrm{vec}(\bm{X}^\top)\mathrm{vec}(\bm{X}^\top)^\top\rangle+\langle \bm{E}_i, \bm{X}\rangle$ in Problem \eqref{prob:lowrankrelaxation_orig}  can be rewritten as the partially separable terms $\frac{1}{2\gamma}\Vert \bm{X}\Vert_F^2+h_i(\bm{X})$ {\color{black}for a common $\gamma$}, where $h_i$ are convex in $\bm{X}$. Then, the optimal value of Problem \eqref{eqn:quadraticShorrelax_lowrank} is at least as large as the relaxation of \cite{bertsimas2021new}
\begin{align}\label{prob:mprt_quadratic}
\min_{\bm{Y} \in \mathrm{Conv}\left(\mathcal{Y}^k_n\right)} \min_{\bm{X} \in \mathbb{R}^{n \times m}, \bm{\theta} \in \mathcal{S}^m_+} \quad & \lambda \cdot \mathrm{tr}(\bm{Y})+\frac{1}{2\gamma}\mathrm{tr}(\bm{\theta})+h(\bm{X})\\
    \text{s.t.} \quad & \frac{1}{2\gamma}\mathrm{tr}(\bm{\theta})+h_i(\bm{X})\leq b_i, \ \forall i \in \mathcal{I},\    
    \ \begin{pmatrix} \bm{\theta} & \bm{X}^\top \\ \bm{X} & \bm{Y}\end{pmatrix} \succeq \bm{0},\nonumber    
\end{align}
\end{proposition}    
\proof{Proof of Proposition \ref{prop:mprt_is_weaker}}
Given the equivalence between Problems \eqref{eqn:quadraticShorrelax_lowrank}--\eqref{eqn:quadraticShorrelax_lowrank2} proven in Theorem~\ref{thm:shorlowrankequiv} {\color{black}and the fact that the quadratic forms are all positive semidefinite}, it suffices to show that the constraints in \eqref{eqn:quadraticShorrelax_lowrank2} imply the constraints in \eqref{prob:mprt_quadratic}. Letting 
$\bm{\theta} := \sum_{i \in [n]} \bm{W}_{x,x}^{(i,i)}$, we observe that $\bm{\theta}$ is feasible for \eqref{prob:mprt_quadratic}. {\color{black}Similarly, convexity of $h, h_i$ implies $\bm{H}-\frac{1}{2\gamma}\bm{I}\succeq \bm{0}$, $\bm{Q}_i-\frac{1}{2\gamma}\bm{I}\succeq \bm{0}$. Therefore, the constraint $\bm{W}_{x,x}\succeq \mathrm{vec}(\bm{X}^\top)\mathrm{vec}(\bm{X}^\top)^\top$ yields the inequalities
\begin{align*}
    \frac{1}{2\gamma}\operatorname{tr}(\bm{\theta})+h(\bm{X})\le &\langle \bm{H},\bm{W}_{x,x}\rangle+\langle \bm{D},\bm{X}\rangle,\\
    \frac{1}{2\gamma}\operatorname{tr}(\bm{\theta})+h_i(\bm{X})\le &\langle \bm{Q}_i,\bm{W}_{x,x}\rangle+\langle \bm{E}_i,\bm{X}\rangle\le b_i, \ \forall i \in \mathcal{I},
\end{align*}
}
which completes the proof.
\hfill\Halmos
\endproof
The proof of Proposition \ref{prop:mprt_is_weaker} reveals that in the boundary case of reduced rank regression (see \S \ref{ssec:ex.RRR}), our lifted relaxation \eqref{eqn:quadraticShorrelax_lowrank2} reduces to a perspective-type relaxation which can be perceived as decomposing the variable $\bm{\theta}$ in \eqref{prob:mprt_quadratic}, and strengthening the relaxation by imposing additional constraints on elements of this decomposition.

\subsection{Projection Cuts for Strengthening the Lifted Relaxation} \label{ssec:strongershor}
{\color{black}In this section, we propose a strategy for improving our compact lifted relaxations, sometimes significantly, by exploiting the fact that the rank constraint $\mathrm{rank}(\bm{X})\leq k$ is inherited by many linear images of $\bm{X}$. In particular, if $\bm{A}, \bm{B}$ are matrices of arbitrary dimension, then the rank constraint implies $\mathrm{rank}(\bm{AXB})\leq k.$}
{\color{black}
Formally, we have the following result:
\begin{theorem}\label{prop:proj_cuts}
Let $\bm{A} \in \mathbb{R}^{p \times n}, \bm{B} \in \mathbb{R}^{m \times q}$ be real matrices, and let us denote $\bm{D} := \bm{A} \otimes \bm{B}^\top \in\mathbb R^{pq\times nm}$. 
In particular, for any $\bm{X} \in \mathbb{R}^{n \times m}$, we have 
$\operatorname{vec}((\bm{AXB})^\top) = \bm{D}\operatorname{vec}(\bm{X}^\top)$. 
Further, for the lifted matrix $\bm{W}_{x,x}\in\mathcal{S}^{nm}$, let the auxiliary matrix $\bm{D}\bm{W}_{x,x}\bm{D}^\top\in\mathcal{S}^{pq}$ be partitioned into $p\times p$ blocks of size $q\times q$. Then, the following constraint is valid for Problem \eqref{prob:lowrankrelaxation_orig}, and may be imposed in \eqref{eqn:quadraticShorrelax_lowrank2}:
\begin{align}\label{eqn:extendedvalidineq}
        \exists \bm{Z}_{D}\in\mathrm{Conv}(\mathcal{Y}^k_p) \quad \text{s.t.} \quad &
        \begin{pmatrix}
        \sum_{a=1}^p (\bm{D}\bm{W}_{x,x}\bm{D}^\top)^{(a,a)} & (\bm{AXB})^\top \\
        \bm{AXB}  & \bm{Z}_{D}
        \end{pmatrix}
        \succeq 0, \bm{AYA^\top} \preceq \rho_{A}\bm{Z}_{D},
\end{align}
where $\rho_{A}:=\lambda_{\max}(\bm{AA^\top})$, and $\bm{Y}$ is the semidefinite matrix in \eqref{eqn:quadraticShorrelax_lowrank2}.
\end{theorem}
\begin{remark} A natural choice for $\bm{B}$ in Theorem \ref{prop:proj_cuts} is taking $q=m$ and $\bm{B}= \bm{I}_m$. Indeed, any variable $\bm{Z}_{D}$ satisfying \eqref{eqn:extendedvalidineq} for $\bm{I}_m$ also satisfies \eqref{eqn:extendedvalidineq} for any $\bm{B} \in \mathbb{R}^{m \times q}$ (simply left/right multiply the first semidefinite inequality by $\operatorname{Diag}(\bm{B}^\top,\bm{I}_p)$/$\operatorname{Diag}(\bm{B},\bm{I}_p)$). The benefit of using a generic $\bm{B}$ (especially one with $q < m$) is computational tractability.
    
\end{remark}
\proof{Proof of Theorem \ref{prop:proj_cuts}}
Let $(\bm{X},\bm{Y})$ be feasible for Problem \eqref{prob:lowrankrelaxation_orig}, and take the exact lifted matrix $\bm{W}_{x,x}=
        \operatorname{vec}(\bm{X}^\top)\operatorname{vec}(\bm{X}^\top)^\top.$ Further, for ease of notation, define $
        \bm{W}_{D}:=
        \bm{D}\bm{W}_{x,x}\bm{D}^\top\in\mathcal{S}^{pq},$ and $
        \bm{\Theta}_{D}(\bm{W}_{x,x})
        :=
        \sum_{a=1}^p \bm{W}_{D}^{(a,a)}\in\mathcal{S}^q.
$
By definition of $\bm{D}$, we have $ \bm{W}_{D} = \operatorname{vec}((\bm{AXB})^\top)\operatorname{vec}((\bm{AXB})^\top)^\top$ and $\bm{\Theta}_{D}(\bm{W}_{x,x}) = \sum_{a=1}^p \mathcal ((\bm{AXB})^\top)_{a,\cdot}^\top ((\bm{AXB})^\top)_{a,\cdot} = (\bm{AXB})^\top\bm{AXB}.$

Let $\bm{Z}_{D}$ be the orthogonal projector onto the column space of $\bm{AY}$. Since $\operatorname{rank}(\bm{AY})\le \operatorname{rank}(\bm{Y})\le k$, the matrix $\bm{Z}_{D}$ has rank at most $k$, so $\bm{Z}_{D}\in\mathcal{Y}_p^k$. 
In addition, from $\bm{Z}_{D}\bm{AY}=\bm{AY}$ and $\bm{X} = \bm{YX}$,
we have the relationship $\bm{Z}_{D}\bm{AXB}=\bm{AXB}$.
Consequently,
\[
        \begin{pmatrix}
        \bm{\Theta}_{D}(\bm{W}_{x,x}) & (\bm{AXB})^\top \\
        (\bm{AXB}) & \bm{Z}_{D}
        \end{pmatrix}
        =
        \begin{pmatrix}
        (\bm{AXB})^\top \bm{AXB} & (\bm{AXB})^\top \\
        \bm{AXB} & \bm{Z}_{D}
        \end{pmatrix} =
        \begin{pmatrix}
        (\bm{AXB})^\top \\
        \bm{Z}_{D}
        \end{pmatrix}
        \begin{pmatrix}
        (\bm{AXB})^\top \\
        \bm{Z}_{D}
        \end{pmatrix}^\top \succeq \bm{0}.
\]
Thus \eqref{eqn:extendedvalidineq} is satisfied by every feasible point of Problem \eqref{prob:lowrankrelaxation_orig} under its exact lift.

For the remaining inequality, observe that 
$\bm{AYA^\top}=\bm{AY}(\bm{AY})^\top$, because $\bm{Y}$ is an orthogonal projection matrix.
By construction, $\operatorname{span}(\bm{Z}_D) = \operatorname{span}(\bm{A} \bm{Y}) = \operatorname{span}(\bm{AYA}^\top)$.
Since $\bm{0}\preceq \bm{AYA^\top}\preceq \bm{AA^\top}\preceq \rho_A \bm{I}_p$, it follows that $\bm{AYA^\top} \preceq \rho_A \bm{Z}_D$. 
Finally, replacing $\bm{Z}_D\in\mathcal{Y}_p^k$ by
$\bm{Z}_D\in\operatorname{Conv}(\mathcal{Y}_p^k)$ preserves validity.
\hfill\Halmos \endproof
}

We support our discussion {\color{black}via Example \ref{example:notequivalent} in Section \ref{ssec:ex1}}, which shows that {\color{black} for a specific low-rank matrix completion problem, \eqref{eqn:quadraticShorrelax_lowrank} gives a relaxation with an optimality gap of around $50\%$, but imposing a small number of projection cuts of the above form closes the gap to $0$. }

\subsection{More Compact Relaxations for Block Diagonal Optimization Problems}\label{ssec:strongerresults_blockdiag}
In this section, we show that if the quadratic objective matrix $\bm{H}$ and the constraint matrices $\bm{Q}_i$ are block diagonal, then we can omit the off-diagonal blocks of $\bm{W}_{x,x}$. We demonstrate the utility of this observation throughout Section \ref{sec:examplesofrelaxations}, by leveraging it repeatedly to eliminate most of the decision variables from Problem \eqref{eqn:quadraticShorrelax_lowrank2} without altering the objective value, for several low-rank problems:
\begin{theorem}
\label{prop:common-core-penalty}
In Problem \eqref{eqn:quadraticShorrelax_lowrank2}, let the quadratic objective matrix $\bm{H}$ be block diagonal, i.e.:
\begin{align*}
        \bm{H}=\operatorname{Diag}(\bm{H}_1,\ldots,\bm{H}_n): \bm{H}_i\in\mathcal{S}^m,
\end{align*}
and let the constraint matrices $\bm{Q}_i$ similarly be block decomposable into $n$ matrices $\bm{Q}_{i,j}$, i.e., 
\begin{align*}
        \bm{Q}_i=\operatorname{Diag}(\bm{Q}_{i,1},\ldots,\bm{Q}_{i,n}): \bm{Q}_{i,j}\in\mathcal{S}^m.
\end{align*}
Then, letting $\bm{X}_{i,.}$ denote a column vector containing the $i$th row of $\bm{X}$, Problem \eqref{eqn:quadraticShorrelax_lowrank2} attains the same optimal value as the following optimization problem:
\begin{align}
\min_{\bm{Y}\in \mathrm{Conv}(\mathcal{Y}^k_n)}\ \min_{\bm{X} \in \mathbb{R}^{n \times m},\bm{S}^1,\ldots,\bm{S}^n \in \mathcal{S}^m_+}
\quad&
\langle \bm{D},\bm{X}\rangle+\sum_{i=1}^n\langle \bm{H}_i,\bm{S}^i\rangle+\lambda\cdot\operatorname{tr}(\bm{Y})\label{eqn:quadpenalty1}\\
\text{s.t.}\quad&
\bm{S}^i\succeq \bm{X}_{i,.}\bm{X}_{i,.}^\top\ i\in[n],\nonumber \
\begin{pmatrix}
\sum_{i=1}^n \bm{S}^i & \bm{X}^\top\\
\bm{X} & \bm{Y}
\end{pmatrix}\succeq \bm{0} ,\nonumber\\
& \sum_{j=1}^n \langle \bm{Q}_{i,j}, \bm{S}^j\rangle+\langle \bm{E}_i, \bm{X}\rangle\leq b_i \ \forall i \in \mathcal{I}.\nonumber
\end{align}
\end{theorem}

{\color{black}

\proof{Proof of Theorem \ref{prop:common-core-penalty}}
  It suffices to show that given any feasible solution to \eqref{eqn:quadpenalty1}, we can construct a feasible solution to \eqref{eqn:quadraticShorrelax_lowrank2} with the same objective value; the converse is immediate. Let $(\bm{X}, \bm{Y}, \bm{S}^i)$ be feasible in \eqref{eqn:quadpenalty1}. Define the block matrix $\bm{W} \in \mathcal{S}^{nm \times nm}_+$ by setting $\bm{W}^{(i,i)}=\bm{S}^{i}$ and $\bm{W}^{(i,j)}=\bm{X}_{i,.} \bm{X}_{j,.}^\top$. Then, it is not hard to see that $\bm{W}-\mathrm{vec}(\bm{X}^\top)\mathrm{vec}(\bm{X}^\top)^\top$ is a block matrix with zero off-diagonal blocks and on-diagonal blocks $\bm{S}^i-\bm{X}_{i,.} \bm{X}_{i,.}^\top\succeq \bm{0}$. Thus, $\bm{W}-\mathrm{vec}(\bm{X}^\top)\mathrm{vec}(\bm{X}^\top)^\top$ is a positive semidefinite matrix, and $\bm{W}\succeq \mathrm{vec}(\bm{X}^\top)\mathrm{vec}(\bm{X}^\top)^\top$. Moreover, due to the block diagonal nature of the lifted semidefinite matrices $\bm{H}$ and $\bm{Q}_i$, $(\bm{X}, \bm{Y}, \bm{W})$ is feasible in \eqref{eqn:quadraticShorrelax_lowrank2} and attains the same objective value.
  \hfill \Halmos
  \endproof
  }

{\color{black}
Compared with \eqref{eqn:quadraticShorrelax_lowrank} or \eqref{eqn:quadraticShorrelax_lowrank2}, the semidefinite relaxation \eqref{eqn:quadpenalty1} involves $n$ positive semidefinite variables of dimension $m$ (vs. one PSD matrix of dimension $nm$ in \eqref{eqn:quadraticShorrelax_lowrank}) and $n+1$ semidefinite constraints of dimension $m+1$ (vs. one of dimension $nm+1$)---both relaxations involve one semidefinite constraint of dimension $n+m$.
}

{\color{black} 
Finally, we now present a restriction of Theorem \ref{prop:proj_cuts} that provides valid inequalities involving only the matrices $\bm{S}^i = \bm{X}_{i,\cdot}\bm{X}_{i,\cdot}^\top$, which can be used for the more compact relaxation \eqref{eqn:quadpenalty1}. 
\begin{corollary} \label{cor:proj_cuts}
Consider a matrix $\bm{A} \in \mathbb{R}^{p \times n}$ where each row and column has at most one non-zero coordinate. Namely, row $j$ of $\bm{A}$ is of the form $\alpha_j \bm{e}_{i(j)}^\top$ for some $\alpha_j \in \mathbb{R}$ and $i : j \mapsto i(j) \in [n]$ is injective.
Then, denoting $\bm{S}^i = \bm{X}_{i,\cdot}\bm{X}_{i,\cdot}^\top, i \in [n]$, 
the following constraint is valid for \eqref{prob:lowrankrelaxation_orig}, and may be imposed in \eqref{eqn:quadpenalty1} :
\begin{align}\label{eqn:extendedvalidineq.compact}
        \exists \bm{Z}_{A}\in\mathrm{Conv}(\mathcal{Y}^k_p) \quad \text{s.t.} \quad &
        \begin{pmatrix}
        \sum_{j \in [p]} \alpha_j^2 \bm{S}^{i(j)} & (\bm{AX})^\top \\
        \bm{AX}  & \bm{Z}_{A}
        \end{pmatrix}
        \succeq 0, \bm{A} \bm{Y} \bm{A}^\top \preceq (\max_{j \in [p]} \alpha_j^2) \,\bm{Z}_{A}.
\end{align}
\end{corollary}

\proof{Proof of Corollary \ref{cor:proj_cuts}} The result follows from applying Theorem \ref{prop:proj_cuts} to $(\bm{A},\bm{I}_m)$. Indeed, in this case, 
$\bm{A}^\top \bm{A} = \sum_{j \in [p]} \alpha_j^2 \bm{e}_{i(j)}\bm{e}_{i(j)}^\top$ and the $i(j)$'s are distinct, so 
\begin{align*}
 \rho_A & = \lambda_{\max}(\bm{AA}^\top) = \lambda_{\max}(\bm{A}^\top \bm{A}) = \max_{j \in [p]} \alpha_j^2, \\
\mbox{ and } 
\sum_{a \in [p]} (\bm{D}\bm{W}_{x,x}\bm{D}^\top)^{(a,a)} & =  (\bm{AXB})^\top\bm{AXB} = \bm{X}^\top \bm{A}^\top \bm{A} \bm{X} = \sum_{j \in [p]} \alpha_j^2 \bm{X} \bm{e}_{i(j)}\bm{e}_{i(j)}^\top \bm{X}^\top = \sum_{j \in [p]} \alpha_j^2 \bm{S}^{i(j)}. \ \Halmos
\end{align*}
\endproof
}

{\color{black}
\subsection{Partial Failure of Lifted Shor Relaxation Without Projection Matrices}\label{ssec:bmfailure}
This subsection\footnote{We are grateful to Amir Ali Ahmadi for an insightful question which prompted this subsection.} justifies the use of a mixed-projection reformulation of Problem \eqref{prob:lowrankquadratic} to design our convex relaxations in Section \ref{sec:lowrankshorrelax}. In the presence of a hard rank constraint only ($\lambda = 0$), it is common to decompose $\bm{X}$ into 
$\bm{X}=\bm{UV^\top}$ 
with $\bm{U} \in \mathbb{R}^{n \times k}, \bm{V} \in \mathbb{R}^{m \times k}$ to enforce $\operatorname{rank}(\bm{X}) \leq k$ by design, as proposed in the literature \citep{burer2003nonlinear}. However, we now show that a naive first-order lift of this reformulation cannot provide a better bound than the trivial convex relaxation obtained by disregarding the rank constraint. 

To obtain a lifted relaxation, we vectorize $\bm{X}, \bm{U}, \bm{V}$ and let the matrix $\bm{W}$ collect the appropriate lifted blocks. Analogously to Proposition \ref{prop:convexrelax}, the corresponding lifted relaxation is then
\begin{equation}
\label{eq:uv-shor-relaxation}
\begin{aligned}
\min_{\substack{\bm{X} \in \mathbb{R}^{n \times m},\bm{U} \in \mathbb{R}^{n \times k}, \\ \bm{V} \in \mathbb{R}^{m \times k},\bm{W}\in \mathcal{S}_+^{n(m+k)+mk}}} \quad
& \langle \bm{H}, \bm{W}_{x,x} \rangle + \langle \bm{D}, \bm{X} \rangle \\
\text{s.t.} \quad
&
\begin{pmatrix}
1 & \mathrm{vec}(\bm{X}^\top) ^\top & \mathrm{vec}(\bm{U})^\top & \mathrm{vec}(\bm{V})^\top \\
\mathrm{vec}(\bm{X}^\top) & \bm{W}_{x,x} & \bm{W}_{x,u} & \bm{W}_{x,v} \\
\mathrm{vec}(\bm{U})  & \bm{W}_{x,u}^\top & \bm{W}_{u,u} & \bm{W}_{u,v} \\
\mathrm{vec}(\bm{V})  & \bm{W}_{x,v}^\top & \bm{W}_{u,v}^\top & \bm{W}_{v,v}
\end{pmatrix}
\succeq 0, \\
& \sum_{j=1}^k \bm{W}_{u,v}^{(j,j)} = \bm{X}, \\ 
& \langle \bm{Q}_i, \bm{W}_{x,x} \rangle + \langle \bm{E}_i, \bm{X} \rangle \le b_i,\ \forall i \in I,
\end{aligned}
\end{equation}
where $\bm{W}$ collects the lifted moment blocks and the constraint $\bm{X} = \sum_{j=1}^k \bm{W}_{u,v}^{(j,j)}$ encodes $\bm{X}=\bm{UV^\top}$. 

We now show that \eqref{eq:uv-shor-relaxation} is no stronger than simply dropping the rank constraint. 
Formally, we have the following result (proof deferred to Section \ref{sec:failureshor1}):
\begin{proposition}\label{prop:failureofshornaive} For any $k \ge 1$, Problem \eqref{eq:uv-shor-relaxation} attains the same optimal value as
\begin{equation}
\label{eq:rank-free-shor-relaxation}
\begin{aligned}
\min_{\bm{X}, \bm{W}_{x,x}} \quad
& \langle \bm{H}, \bm{W}_{x,x} \rangle + \langle \bm{D}, \bm{X} \rangle \\
\text{s.t.} \quad
&
\begin{pmatrix}
1 & \mathrm{vec}(\bm{X}^\top)^\top \\
\mathrm{vec}(\bm{X}^\top) & \bm{W}_{x,x}
\end{pmatrix}
\succeq 0, \\
& \langle \bm{Q}_i, \bm{W}_{x,x} \rangle + \langle \bm{E}_i, \bm{X} \rangle \le b_i,
\qquad \forall i \in I.
\end{aligned}
\end{equation}
\end{proposition}



Proposition \ref{prop:failureofshornaive} shows that the first-order lifting of the factorized formulation $\bm{X} = \bm{U}\bm{V}^\top$ does not encode any nontrivial rank
information: once the rank-one condition on the moment matrix is dropped, the blocks involving $\bm{U}$ and $\bm{V}$ can always be completed so as to represent an
arbitrary matrix $\bm{X}$. In other words, Proposition \ref{prop:failureofshornaive} holds because there are no additional constraints on $\bm{U}, \bm{V}$. In contrast, our relaxation starts from the mixed-projection reformulation \eqref{prob:lowrankrelaxation_orig}, which relies on a structured decomposition of $\bm{X}$.

In particular, we can recover a non-trivial compact semidefinite relaxation from a decomposition $\bm{X}=\bm{UV}^\top$ if we further require that $\bm{U}^\top \bm{U}=\bm{I}_k$. In this case, we must have $\bm{X}^\top \bm{X}=\bm{VV^\top}$. These two relationships translate into additional constraints in the lifted semidefinite relaxation: $\mathrm{tr}(\bm{W}_{u,u}^{(j,l)})=\delta_{j,l} \ \forall j,l \in [k]$ and $\sum_{i \in [n]}\bm{W}_{x,x}^{(i,i)}=\sum_{j \in [k]}\bm{W}_{v,v}^{(j,j)}$, respectively. With these additional constraints, we recover our semidefinite relaxation \eqref{eqn:quadraticShorrelax_lowrank2} 
(proof deferred to Section \ref{ssec:shorequiv3}):
\begin{proposition}\label{prop:shorequiv2}
    Let $\lambda=0$. The following problem attains the same optimal value as \eqref{eqn:quadraticShorrelax_lowrank2}:
\begin{equation}
\label{eq:uv-shor-relaxation2}
\begin{aligned}
\min_{\substack{\bm{X} \in \mathbb{R}^{n \times m},\bm{U} \in \mathbb{R}^{n \times k}, \\ \bm{V} \in \mathbb{R}^{m \times k},\bm{W}\in \mathcal{S}_+^{n(m+k)+mk}}} \quad
& \langle \bm{H}, \bm{W}_{x,x} \rangle + \langle \bm{D}, \bm{X} \rangle \\
\text{s.t.} \quad
&
\begin{pmatrix}
1 & \mathrm{vec}(\bm{X}^\top) ^\top & \mathrm{vec}(\bm{U})^\top & \mathrm{vec}(\bm{V})^\top \\
\mathrm{vec}(\bm{X}^\top) & \bm{W}_{x,x} & \bm{W}_{x,u} & \bm{W}_{x,v} \\
\mathrm{vec}(\bm{U})  & \bm{W}_{x,u}^\top & \bm{W}_{u,u} & \bm{W}_{u,v} \\
\mathrm{vec}(\bm{V})  & \bm{W}_{x,v}^\top & \bm{W}_{u,v}^\top & \bm{W}_{v,v}
\end{pmatrix}
\succeq 0, \\
& \sum_{j=1}^k \bm{W}_{u,v}^{(j,j)} = \bm{X}, \mathrm{tr}(\bm{W}_{u,u}^{(j,l)})=\delta_{j,l} \ \forall j,l \in [k], \\
& \sum_{i \in [n]}\bm{W}_{x,x}^{(i,i)}=\sum_{j \in [k]}\bm{W}_{v,v}^{(j,j)},\\
& \langle \bm{Q}_i, \bm{W}_{x,x} \rangle + \langle \bm{E}_i, \bm{X} \rangle \le b_i,\ \forall i \in I.
\end{aligned}
\end{equation}
\end{proposition}
The proof of Proposition \ref{prop:shorequiv2} reveals that, upon imposing $\bm{X}=\bm{UV^\top}, \bm{U}^\top \bm{U}=\bm{I}_k, \bm{X^\top X}=\bm{VV^\top}$, the block matrix $\sum_{j \in [k]}\bm{W}_{u,u}^{(j,j)}$ dominates the relaxed projection variable $\bm{Y}=\bm{X} \bm{\Theta}^\dagger \bm{X}^\top: \bm{\Theta}=\sum_{j \in [k]}\bm{W}_{v,v}^{(j,j)}$. All in all, a lifted relaxation of a \cite{burer2003nonlinear} factorization 
recovers relaxation \eqref{eqn:quadraticShorrelax_lowrank2} in extended form, provided additional structure on the factors $\bm{U},\bm{V}$ is imposed.  

\begin{remark} One can show that Proposition \ref{prop:failureofshornaive} holds even when adding the constraints $\mathrm{tr}(\bm{W}_{u,u}^{(j,l)})=\delta_{j,l} \ \forall j,l \in [k]$. To obtain a non-trivial relaxation in Proposition \ref{prop:shorequiv2}, it is crucial to observe that $\bm{U}^\top \bm{U}=\bm{I}_k$ implies  $\bm{X}^\top \bm{X}=\bm{VV^\top}$ and to use the latter constraints to relate the diagonal blocks of $\bm{W}_{x,x}$ (which appears in the objective) with those of $\bm{W}_{v,v}$.
\end{remark}
}

\section{Examples of Low-Rank Relaxations}\label{sec:examplesofrelaxations}
This section applies the lifted relaxation proposed in Section \ref{sec:lowrankshorrelax} to two important problems from the low-rank literature: (i) matrix completion (Section \ref{ssec:ex.mc}), which possesses a row-separable quadratic loss, and (ii) reduced rank regression (Section \ref{ssec:ex.RRR}), which possesses an objective dependent explicitly on $\bm{XX^\top}$. By exploiting {\color{black}their respective} problem structure, 
we show that our lifted relaxation can be further reduced to an equivalent relaxation with {\color{black}no semidefinite matrices or constraints of dimension larger than the sum of the two dimensions of the low-rank matrix}, as opposed to the $O((n+m)^2)$ dimension blocks in our lifted relaxations from the previous section. 

Section \ref{append.additionalexamples} supports this discussion by applying our lifted relaxation technique to three additional families of low-rank problems: basis pursuit, sparse plus low-rank matrix decomposition, and low-rank factor analysis. 

\subsection{Matrix Completion} \label{ssec:ex.mc}

Given a random sample $\left\{A_{i,j}: (i,j) \in \Omega \subseteq [n] \times [m]\right\}$ of a matrix $\bm{A} \in \mathbb{R}^{n \times m}$, the goal of the low-rank matrix completion problem is to reconstruct the matrix $\bm{A}$, by assuming it is approximately low-rank \citep{candes2009exact}. This problem admits the formulation:  
\begin{align}\label{prob:lrmc_orig}
    \min_{\bm{Y} \in \mathcal{Y}_n^k }\min_{\bm{X}\in \mathbb{R}^{n \times m}} \quad & \Vert \mathcal{P}(\bm{A})-\mathcal{P}(\bm{X})\Vert_F^2+\lambda \cdot \mathrm{tr}(\bm{Y}) \quad \text{s.t.} \quad \bm{X}=\bm{Y}\bm{X},
\end{align}
where $\lambda\geq 0$ is a penalty multiplier on the rank of $\bm{X}$ through the trace of $\bm{Y}$, $k$ is a hard constraint on the rank of $\bm{X}$ through the trace of $\bm{Y}$, and $$\mathcal{P}(\bm{A})_{i,j}=\begin{cases}A_{i,j} \quad & \text{if} \ (i,j) \in \Omega \\ 0 \quad & \text{otherwise}\end{cases}$$ is a linear map which masks the hidden entries of $\bm{A}$ by outputting a matrix equal to $\bm{A}$ on $\Omega$ and $0$ otherwise. By expanding the quadratic $\Vert \mathcal{P}(\bm{A})-\mathcal{P}(\bm{X})\Vert_F^2=\sum_{i \in [n]}\left(\sum_{j \in [m]: (i,j) \in \Omega}(X_{i,j}-A_{i,j})^2\right)$, we can invoke {\color{black}Theorem \ref{prop:common-core-penalty}} to obtain the following relaxation of \eqref{prob:lrmc_orig}
\begin{equation}\label{prob:lowrankrelaxation_reduced}
\begin{aligned}
    \min_{\bm{Y} \in \mathrm{Conv}(\mathcal{Y}_n^k) } \min_{\bm{X}\in \mathbb{R}^{n \times m}, \bm{S}^i \in \mathcal{S}^{m}_+} \quad & \sum_{i \in [n]}\langle \bm{S}^{i}, \bm{H}^i\rangle-2\left\langle \mathcal{P}(\bm{X}), \mathcal{P}(\bm{A})\right\rangle
    +\left\langle \mathcal{P}(\bm{A}), \mathcal{P}(\bm{A})\right\rangle
    +\lambda \cdot \mathrm{tr}(\bm{Y}) \\
    \text{s.t.} \quad & \bm{S}^i\succeq \bm{X}_{i,.} \bm{X}_{i,.}^\top, \ 
    \begin{pmatrix} \sum_{i \in [n]}\bm{S}^{i} & \bm{X}^\top \\ \bm{X} & \bm{Y}\end{pmatrix} \succeq \bm{0}.
\end{aligned}
\end{equation}
where $\bm{H}^{i}$ is a diagonal matrix which takes entries $\bm{H}^i_{j,j}=1$ if $(i,j) \in \Omega$ and $\bm{H}^i_{j,j}=0$ otherwise, 
and we retain the constant term $\langle \mathcal{P}(\bm{A}), \mathcal{P}(\bm{A})\rangle$ to ensure that all relaxations are directly comparable.
Compared with the matrix perspective relaxation of \citet{bertsimas2021new}, our relaxation is directly applicable to \eqref{prob:lrmc_orig}, while \citet{bertsimas2021new} requires the presence of an additional Frobenius regularization term $+\frac{1}{2\gamma}\Vert \bm{X}\Vert_F^2$ in the objective.

{\color{black} Because each matrix $\bm{H}^i$ in Problem \eqref{prob:lowrankrelaxation_reduced} is binary, we can eliminate the individual matrices $\bm{S}^i$ and keep only one additional semidefinite variable $\bm{\Theta} \in \mathcal{S}^m_+$ to encode their sum. Formally, we have the following result (proof deferred to Section \ref{ssec:redmc2prop}):
\begin{proposition}\label{prop:reducedmatrixcompletionii}
Define $\bar{\bm{H}}:=\operatorname{Diag}(\bar{\bm{h}})$ where $\bar h_j:=\min_{i\in[n]} H^i_{j,j}$. 
Then, Problem \eqref{prob:lowrankrelaxation_reduced} attains the same optimal value as the following optimization problem:
\begin{equation}\label{eqn:lmrc_further_reduced}
\begin{aligned}
\min_{\bm{Y}\in\mathrm{Conv}(\mathcal{Y}_{n}^k)}
\min_{\bm{X}\in\mathbb{R}^{n\times m},  \bm{\Theta}\in\mathcal{S}^m_{+}}
\quad&
\langle \bar{\bm{H}},\bm{\Theta}\rangle
+
\sum_{i\in[n]}\langle \bm{H}^i-\bar{\bm{H}},\bm{X}_{i,.}\bm{X}_{i,.}^\top\rangle
-2\langle \mathcal{P}(\bm{X}),\mathcal{P}(\bm{A})\rangle+\langle \mathcal{P}(\bm{A}),\mathcal{P}(\bm{A})\rangle
+\lambda\,\operatorname{tr}(\bm{Y})\\
\textnormal{s.t.}\quad&
\begin{pmatrix}
\bm{\Theta} & \bm{X}^\top\\
\bm{X} & \bm{Y}
\end{pmatrix}\succeq \bm{0}.
\end{aligned}
\end{equation}
\end{proposition}

We point out that if $\bar{\bm H} = \bm{0}$, relaxation \eqref{eqn:lmrc_further_reduced} is weak. Indeed, in this case, 
any completion that satisfies $\mathcal{P}(\bm{X}-\bm{A})=\bm{0}, \bm{Y}=\epsilon \bm{I}, \bm{\Theta}=\epsilon^{-1} \bm{X}^\top \bm{X}$ is feasible and attains objective value $0$ as $\epsilon \rightarrow 0$. This situation ($\bar{\bm H} = \bm{0}$) occurs if each column is missing for at least one row (meaning $\min_{i \in [n]} H_{j,j}^i=0$).
In other words, to obtain nontrivial relaxations via the compact unregularized relaxation \eqref{eqn:lmrc_further_reduced}, there must exist at least one column $j$ which is observed in every row. When $\bm{X}$ represents a dataset of $n$ $p$-dimensional observations, this means that there must exist at least one variable which is consistently available for all observations.  
For example, \citet{goldberg2001eigentaste} consider a collaborative filtering setting where all users are provided with universal queries before they begin rating other items and \citet{athey2021matrix} consider a panel data setting possibly involving complete rows for control units where the goal is to infer missing elements of a control outcome matrix. 

On the other hand, to obtain non-trivial relaxations when $\bar{\bm H} = \bm{0}$, one must either add a Frobenius regularization term to the objective or strengthen the relaxation via projection cuts (or both).
}

With the additional regularization term $+\frac{1}{2\gamma}\Vert \bm{X}\Vert_F^2$ in the objective, our approach leads to relaxations of the form \eqref{prob:lowrankrelaxation_reduced} after redefining $\bm{H}^i \leftarrow \bm{H}^i + \frac{1}{2\gamma}\bm{I}_m$. Per Proposition \ref{prop:mprt_is_weaker}, our resulting relaxations are at least as strong as the relaxation of \citet{bertsimas2021new}. {\color{black}However, the matrices $\bm{H}^i$ are no longer binary and the compact reformulation in Proposition \ref{prop:reducedmatrixcompletionii} no longer holds.}

{\color{black}
Applying Corollary \ref{cor:proj_cuts}, 
Problem \eqref{prob:lowrankrelaxation_reduced} can be strengthened by imposing {projection cuts}. Let $\mathcal{R}=\{r_1,\ldots,r_{|R|}\}$ be a nonempty subset of the rows of $\bm{X}$ of cardinality $\vert \mathcal{R}\vert \geq k+1$, let  $\bm{E}_{\mathcal{R}} \in \{0, 1\}^{\vert \mathcal{R}\vert \times n}$ denote the corresponding row-selection matrix, and define $\bm{X}_{\mathcal{R}}:=\bm{E}_{\mathcal{R}}\bm{X} \in \mathbb{R}^{\vert \mathcal{R}\vert \times m}$. Then, since $\operatorname{rank}(\bm{X}_{\mathcal{R}})\leq \operatorname{rank}(\bm{X})$ and $\bm{E}_{\mathcal{R}} \bm{E}_{\mathcal{R}}^\top=\bm{I}_{|R|}$, we have the following auxiliary valid inequalities for Problem \eqref{prob:lowrankrelaxation_reduced}:
\begin{align}\label{eqn:projcut_matcomp}
\exists \bm{Z}_{\mathcal{R}}\in \operatorname{Conv}(\mathcal Y_{|R|}^k)
\ \textnormal{s.t.}\
\begin{pmatrix}
\sum_{i\in \mathcal{R}}\bm{S}^i &\bm{X}_\mathcal{R}^\top\\
\bm{X}_\mathcal{R} & \bm{Z}_\mathcal{R}
\end{pmatrix}\succeq \bm{0}, \bm{E}_{\mathcal{R}}\bm{Y}{\bm{E}_\mathcal{R}}^\top \preceq \bm{Z}_\mathcal{R}.
\end{align}
We require that $\vert \mathcal{R}\vert > k$ in order that one cannot take $\bm{Z}_\mathcal{R}=\bm{I}$, rendering the cut trivial.
}

\subsection{Reduced Rank Regression}\label{ssec:ex.RRR}
Given a response matrix $\bm{B} \in \mathbb{R}^{n\times m}$ and a predictor matrix $\bm{A}\in \mathbb{R}^{\color{black}n\times p}$, an important problem in high-dimensional statistics is to recover a low-complexity model which relates the matrices $\bm{B}$ and $\bm{A}$. A popular choice for doing so is to assume that $\bm{B}, \bm{A}$ are related via $\bm{B}=\bm{A}\bm{X}+\bm{E}$, where $\bm{X} \in \mathbb{R}^{p \times {\color{black}m}}$ is a coefficient matrix, $\bm{E}$ is a matrix of noise, and we require that the rank of $\bm{X}$ is small so that the linear model is parsimonious \citep{negahban2011estimation}. This yields:
\begin{align}\label{eqn:rrr_orig}
    \min_{\bm{X} \in \mathbb{R}^{p \times {\color{black}m}}} \quad \Vert \bm{B}-\bm{A}\bm{X}\Vert_F^2+\mu \cdot \mathrm{rank}( \bm{X}),
\end{align}
where $\mu >0$ controls the complexity of the estimator. 
For this problem, our lifted relaxation \eqref{eqn:quadraticShorrelax_lowrank2} is equivalent to the (improved) matrix perspective relaxation of \citet{bertsimas2021new}.

Indeed, by invoking Theorem \ref{thm:shorlowrankequiv}--on $\mathrm{vec}(\bm{X})\mathrm{vec}(\bm{X})^\top$ rather than the mathematically equivalent $\mathrm{vec}(\bm{X}^\top)\mathrm{vec}(\bm{X}^\top)^\top$ for notational convenience--we obtain \eqref{eqn:rrr_orig}'s lifted relaxation
\begin{equation}\label{prob:shorrelaxation_rrr}
\begin{aligned}
    \min_{\bm{Y} \in \mathrm{Conv}(\color{black}\mathcal{Y}_m) } \min_{\bm{X}\in \mathbb{R}^{p \times m}, \bm{W} \in \mathcal{S}^{pm}_+} \quad & \left\langle \bm{A}^\top \bm{A}, \sum_{\color{black}i \in [m]}\bm{W}^{(i,i)}\right\rangle+\langle \bm{B}, \bm{B}\rangle-2\langle \bm{AX}, \bm{B}\rangle+\mu\cdot \mathrm{tr}(\bm{Y}) \\
    \text{s.t.} \quad & \bm{W}\succeq \mathrm{vec}(\bm{X})\mathrm{vec}(\bm{X})^\top, \ \begin{pmatrix} \sum_{\color{black}i \in [m]}\bm{W}^{(i,i)} & \bm{X} \\ \bm{X}\color{black}^\top & \bm{Y}\end{pmatrix} \succeq \bm{0},
\end{aligned}
\end{equation}
for which we {\color{black}have the following corollary to Theorem \ref{prop:common-core-penalty}}: 
\begin{corollary}\label{corr:rrr}
Problem \eqref{prob:shorrelaxation_rrr} attains the same objective value as
\begin{equation}\label{prob:shorrelaxation_rrr_compact}
\begin{aligned}
    \min_{\bm{Y} \in \mathrm{Conv}(\color{black}\mathcal{Y}_m) } \min_{\bm{X}\in \mathbb{R}^{p \times m}, \bm{\theta} \in \mathcal{S}^{p}_+} \quad & \left\langle \bm{A}^\top \bm{A}, \bm{\theta}\right\rangle+\langle \bm{B}, \bm{B}\rangle-2\langle \bm{AX}, \bm{B}\rangle+\mu\cdot \mathrm{tr}(\bm{Y}) \\
    \text{s.t.} \quad & \begin{pmatrix} \bm{\theta} & \bm{X} \\ \bm{X}{\color{black}^\top} & \bm{Y}\end{pmatrix} \succeq \bm{0},
\end{aligned}  
\end{equation}
which corresponds to the improved relaxation of \citet[][Equation 7]{bertsimas2021new}.
\end{corollary}

Corollary \ref{corr:rrr} suggests other low-rank problems which are quadratic through $\bm{X} \bm{X}{\color{black}^\top}$ (or $\bm{X}{\color{black}^\top} \bm{X}$), e.g., low-rank factor analysis \citep[][]{bertsimas2017certifiably}, sparse plus low-rank matrix decompositions \citep[][]{bertsimas2021sparse} and quadratically constrained programming \citep{wang2022tightness} admit similarly compact lifted relaxations.

{\color{black}In this section, we have leveraged Theorem \ref{prop:common-core-penalty} to show that for two quadratic low-rank problems (plus three additional examples in Section \ref{append.additionalexamples}), it is possible to eliminate enough variables in the lifted relaxation that no matrices of size $n^2 \times n^2$ remain. 
That is, while naive lifted relaxations may appear to be too large to be useful, they can often be reduced to practical sizes.}

\section{Numerical Results}\label{sec:numerics}
{\color{black}In this section, we report four numerical experiments on low-rank matrix completion problems. Sections \ref{ssec:4.1}–\ref{ssec:4.3} use synthetic instances to isolate the quality of our bounds compared to the matrix perspective relaxation of \citet{bertsimas2021new}, the impact of the projection cuts developed in this paper, and the scalability of our approach, respectively. Section \ref{ssec:real.data} then evaluates the same modeling ideas on real datasets with native missingness patterns.

Our first experiment was conducted on a MacBook Pro laptop (Apple M3, $36$ GB), using MOSEK version 10.2, Julia version $1.9$, and JuMP.jl version $1.13.0$. All other experiments were conducted on the Imperial High Performance Computing environment (AMD EPYC 7742, 1 TB RAM) using MOSEK version 11.0.30, Julia version $1.11$, and JuMP.jl version 1.29.3. All solver parameters are set to their default values for all experiments except where explicitly stated otherwise. 

}

All synthetic experiments in this section use the data generation process of \cite{candes2009exact}: We construct a matrix of observations, $\bm{A}_{\text{full}} \in \mathbb{R}^{n \times m}$, from a rank-$k$ model: $\bm{A}_{\text{full}} = \bm{U} \bm{V} + \epsilon \, \bm{Z}$, where the entries of $\bm{U} \in \mathbb{R}^{n \times k}, \bm{V} \in \mathbb{R}^{k \times m}$, and $\bm{Z} \in \mathbb{R}^{n \times m}$ are drawn independently from a standard normal distribution, and $\epsilon \geq 0$ models the degree of noise. We fix $\lambda=0, \epsilon = 0.1$ and set $n,m, k$ as indicated in each experiment. We then sample a random subset $\Omega \subseteq [n] \times [m]$ of predefined size $\vert \Omega\vert$ uniformly at random \citep[see also][section 1.1.2]{candes2009exact}. Each reported result is averaged over $10$ random seeds.

\subsection{Comparison of Lifted Relaxations With Matrix Perspective Relaxation}\label{ssec:4.1}

We first evaluate the quality of our new relaxations, compared with the matrix perspective relaxation of \citet[][MPRT]{bertsimas2021new}. 
We consider the regularized problem
\begin{align*}
    \min_{\bm{X} \in \mathbb{R}^{n \times m}}\quad \frac{1}{2\gamma}\Vert \bm{X}\Vert_F^2 +\frac{1}{2}\sum_{(i,j) \in \Omega}(A_{i,j}-X_{i,j})^2 \ \text{s.t.} \ \mathrm{rank}(\bm{X})\leq k,
\end{align*}
for some regularization parameter $\gamma>0$, with $m=n=8$ and $k=2$. As $\gamma \rightarrow \infty$, we recover the solution of \eqref{prob:lrmc_orig}. {\color{black}
We compare the lower bounds obtained by three different approaches: MPRT, our lifted relaxation \eqref{prob:lowrankrelaxation_reduced}, hereafter denoted ``Lifted'', and our lifted relaxation augmented with the projection cuts \eqref{eqn:projcut_matcomp} for all three-row subsets $\mathcal{R} \subseteq [n]$, i.e., $\vert \mathcal{R}\vert=3$ (``Lifted-proj-cuts''). Figure \ref{fig:gw_vary_gamma} reports the relative optimality gap 
achieved by each approach---using the alternating minimization method of \cite{burer2003nonlinear, jain2013low} initialized with a truncated SVD of $\mathcal{P}(\bm{A})$ for the upper bound (absolute values are reported in Figures \ref{fig:gw_vary_gamma_absoluteLBs}--\ref{fig:gw_vary_gamma_absoluteUBs_BM})---as $\gamma$ increases, for different proportions of sampled entries $p =|\Omega|/mn$. We compute the relative gap as $\mathrm{Gap}=\frac{UB-LB}{UB}$ with $UB>0$.

Consistent with Proposition \ref{prop:mprt_is_weaker}, we observe that Lifted and Lifted-proj-cuts obtain smaller optimality gaps than MPRT, for all values of $\gamma$, and that the benefit increases as the fraction of sampled entries $p$ increases. In particular, when $p=0.95$, MPRT achieves an uninformative gap of $100 \%$ as $\gamma$ increases (by returning a trivial lower bound of 0, see Figure \ref{fig:gw_vary_gamma_absoluteLBs}) while Lifted achieves a gap lower than 50\% for the largest values of $\gamma$. 
Lifted-proj-cuts has an average optimality gap of $0$ for all values of $\gamma$ (average absolute gap $<10^{-5}$ for all $\gamma$). From this experiment, it seems that imposing the projection cuts leads to significantly tighter bounds in some circumstances ($p=0.95$), and only moderately tighter bounds in others ($p=0.5$). The computational cost of imposing projection cuts is also rather modest at this scale; an average solve time of around $0.3$ seconds per instance with projection cuts compared with around $0.05$ seconds without projection cuts at this scale (see Figure \ref{fig:gw_vary_gamma_runtimes} for computational times). }

\begin{figure}
\begin{subfigure}{0.45\textwidth}
        \includegraphics[width=\textwidth]{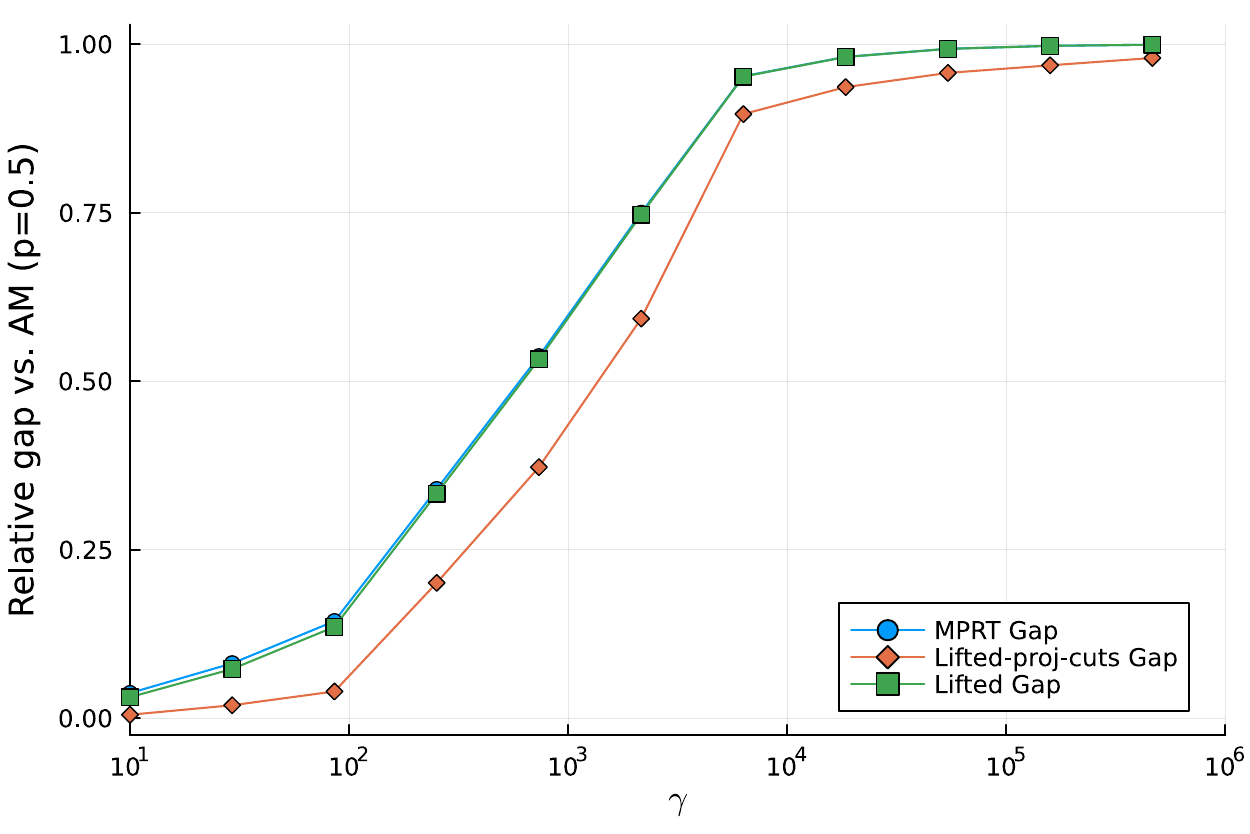}
        \caption{$p=0.5$}
        \label{fig:bounds_0p5}
    \end{subfigure}
        \begin{subfigure}{0.45\textwidth}
        \includegraphics[width=\textwidth]{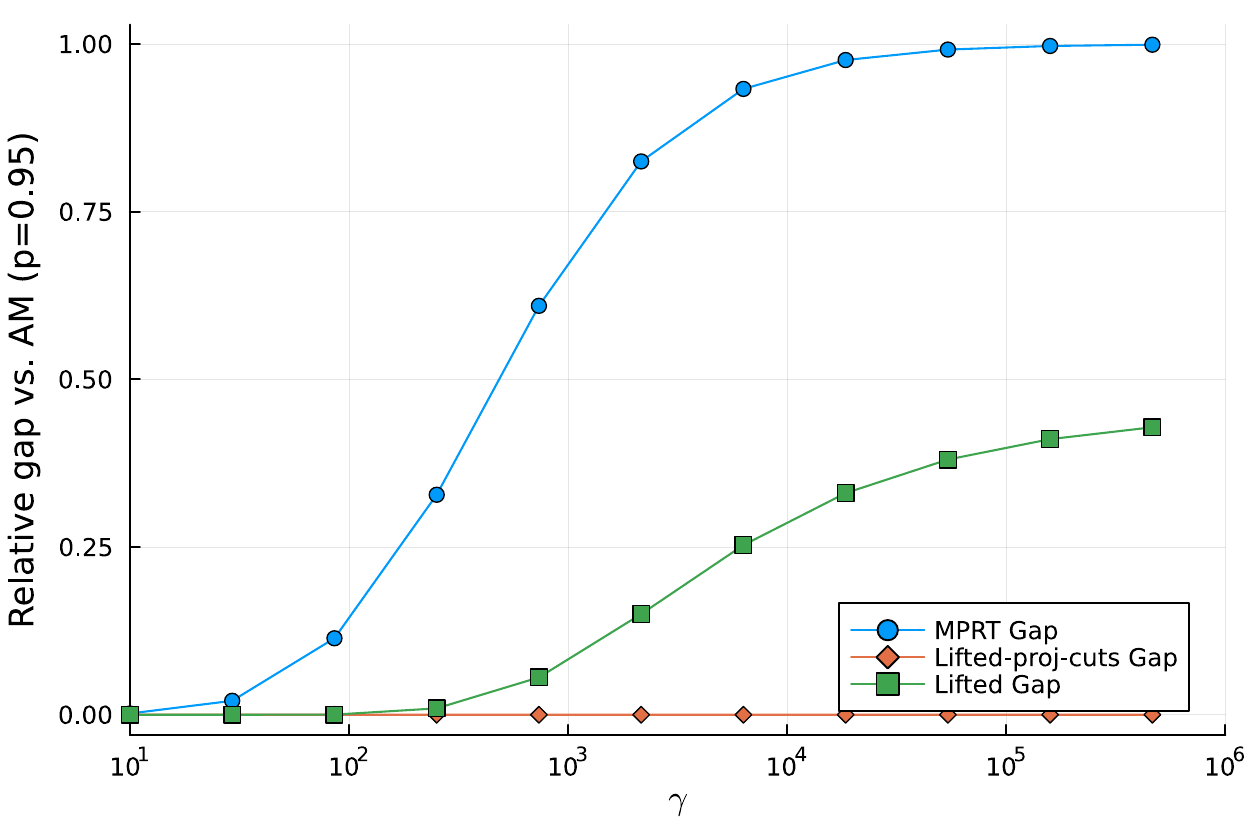}
        \caption{$p=0.95$ }
        \label{fig:bounds_0p95}
    \end{subfigure}
    \caption{Relative gap obtained with different relaxations of the regularized matrix completion problem as we vary $\gamma$. We fix $n=8$. Results are averaged over 10 replications.
    }
    \label{fig:gw_vary_gamma}
\end{figure}

\subsection{Performance of Projection Cuts on Synthetic Data}
{\color{black}
Our second experiment isolates the effect of projection cuts on the strength and computational cost of our lifted relaxation \eqref{prob:lowrankrelaxation_reduced}. We use the same experimental setup as Figure \ref{fig:gw_vary_gamma}, with varying $n, k, \gamma$ and fixing $m=10, \epsilon=0.1, \frac{\vert\Omega\vert}{nm}=0.5$. For each triple $(n, k, \gamma)$ we compare three lower bounds: the lifted relaxation \eqref{prob:lowrankrelaxation_reduced} without projection cuts, the same relaxation with $100$ randomly selected projection cuts \eqref{eqn:projcut_matcomp} of cardinality $k+1$, and the same relaxation with all projection cuts of cardinality $k+1$. If fewer than $100$ cuts are available, the random cut variant imposes all of them. We report average runtimes and average relative optimality gaps over $10$ instances per triple $(n, k, \gamma)$, using the same alternating minimization setup as 
in Figure \ref{fig:gw_vary_gamma} to compute upper bounds. All SDP relaxations are solved with a relative optimality tolerance of $10^{-8}$ and a time limit of $16$ hours. For SDP instances that return the status \verb|SLOW_PROGRESS| (approximately $25\%$ of instances), we use the objective bound when \verb|Mosek| terminates as a surrogate for the semidefinite bound.

\begin{table}[t]
\centering
\caption{\color{black} Average performance of lifted relaxation \eqref{prob:lowrankrelaxation_reduced} with and without the projection cuts \eqref{eqn:projcut_matcomp} on synthetic $n \times 10$ matrix completion instances. ``No. cuts'' denotes the relaxation \eqref{prob:lowrankrelaxation_reduced}, ``Rand.'' denotes \eqref{prob:lowrankrelaxation_reduced} augmented with 100 uniformly sampled projection cuts \eqref{eqn:projcut_matcomp} of cardinality $k+1$, while ``All'' denotes \eqref{prob:lowrankrelaxation_reduced} augmented with all such cuts. Gaps are relative gaps $(\mathrm{UB}-\mathrm{LB})/\mathrm{UB}$, averaged over 10 replications; runtimes are in seconds. The reported regularization parameter is $\gamma/nm$. A dash denotes that an instance could not be solved within the given time/memory budget.}
\label{tab:projection-cut-experiment}
\scriptsize
\begin{tabular}{cc r r r r r r r r r r r r}
\toprule
& & \multicolumn{6}{c}{$k=2$} & \multicolumn{6}{c}{$k=3$} \\
\cmidrule(lr){3-8}\cmidrule(l){9-14}
& & \multicolumn{2}{c}{No cuts} 
& \multicolumn{2}{c}{Rand. cuts} 
& \multicolumn{2}{c}{All cuts} 
& \multicolumn{2}{c}{No cuts} 
& \multicolumn{2}{c}{Rand. cuts} 
& \multicolumn{2}{c}{All cuts} \\
\cmidrule(lr){3-4}\cmidrule(lr){5-6}\cmidrule(lr){7-8}
\cmidrule(lr){9-10}\cmidrule(lr){11-12}\cmidrule(l){13-14}
$n$ & $\gamma/nm$ 
& Gap (\%) & T (s) 
& Gap (\%) & T (s) 
& Gap (\%) & T (s) 
& Gap (\%) & T (s) 
& Gap (\%) & T (s) 
& Gap (\%) & T (s) \\
\midrule
10 & $10^4$ & 16.60\% & 3.34 & 4.98\%  & 1.47 & 1.29 \% & 1.86 & 9.79\% & 3.37 & 7.91\%  & 2.25 & 7.06 \% & 6.24 \\
10 & $10^5$ & 53.01\% & 3.38 & 7.99\%  & 1.63 & 6.74 \% & 2.16 & 19.80\% & 3.37 & 17.47\%  & 1.97 & 15.73 \% & 5.13 \\
10 & $10^6$ & 89.56\% & 3.38 & 27.97\%  & 1.89 & 25.83 \% & 2.11 & 41.74\% & 3.32 & 39.48\%  & 1.90 & 37.40 \% & 4.96 \\\midrule
20 & $10^4$ & 0.93\% & 4.28 & 0.04\%  & 16.61 & 0.00 \% & 113.75 & 2.74\% & 4.19 & 1.27\%  & 34.76 & 0.02 \% & 4234.1 \\
20 & $10^5$ & 25.17\% & 4.20 & 2.45\%  & 23.50 & 0.00 \% & 120.29 & 17.32\% & 4.45 & 12.55\%  & 38.26 & 0.49 \% & 4779.1 \\
20 & $10^6$ & 79.53\% & 4.82 & 21.95\%  & 23.46 & 0.00 \% & 131.48 & 64.27\% & 4.20 & 56.13\%  & 34.69 & 12.77 \% & 5439.7 \\
\midrule
40 & $10^4$ & 0.00\% & 11.64 & 0.00\%  & 1562.3 & --  & -- & 0.02\% & 23.98 & 0.00\%  & 1891.7 & -- & -- \\
40 & $10^5$ & 0.63\% & 13.98 & 0.00\%  & 1641.4 & --  & -- & 1.36\% & 13.80 & 0.32\%  & 2751.7 & -- & -- \\
40 & $10^6$ & 43.23\% & 13.21 & 12.49\%  & 2439.5 & -- & -- & 35.10\% & 13.02 & 26.95\%  & 2948.2 & -- & -- \\\midrule
60 & $10^4$ & 0.00\% & 105.65 & 0.00\%  & 20375.3 & -- & -- & 0.00\% & 108.16 & --  & -- & -- & -- \\
60 & $10^5$ & 0.00\% & 118.27 & 0.00\%  & 19914.8 & -- & -- & 0.00\% & 131.56 & --  & -- & -- & -- \\
60 & $10^6$ & 20.46\% & 117.87 & 7.62\%  & 42081.5 & -- & -- & 15.41\% & 119.30 & --  & -- & -- & -- \\\midrule
80 & $10^4$ & 0.00\% & 467.88 & --  & -- & -- & -- & 0.00\% & 496.72 & --  & -- & -- & -- \\
80 & $10^5$ & 0.00\% & 808.08 & --  & -- & -- & -- & 0.00\% & 467.07 & --  & -- & -- & -- \\
80 & $10^6$ & 5.09\% & 603.24 & --  & -- & -- & -- & 4.28\% & 568.44 & --  & -- & -- & -- \\\midrule
100 & $10^4$ & 0.00\% & 1720.5 & --  & -- & -- & -- & 0.00\% & 1577.2 & --  & -- & -- & -- \\
100 & $10^5$ & 0.00\% & 1995.8 & --  & -- & -- & -- & 0.00\% & 1803.1 & --  & -- & -- & -- \\
100 & $10^6$ & 0.46\% & 2521.3 & --  & -- & -- & -- & 0.51\% & 2747.1 & --  & -- & -- & -- \\
\bottomrule
\end{tabular}
\end{table}
Table \ref{tab:projection-cut-experiment} shows that the base relaxation \eqref{prob:lowrankrelaxation_reduced} is very strong when $n$ is large or when $\gamma$ is relatively small (i.e., the amount of regularization is large), but its gap increases as $\gamma$ increases or as $n$ decreases. 
Furthermore, we observe that projection cuts can substantially strengthen
the lifted relaxation \eqref{prob:lowrankrelaxation_reduced}, especially in the weakly regularized regime where $\gamma$ is very large. 
In this case, imposing all cardinality-$k+1$ projection cuts often closes most of the remaining gap, at the cost of a substantially larger SDP. Moreover, the random cut variant captures a meaningful portion of this performance improvement while remaining much cheaper than imposing all cuts. 
This supports a simple strategy: solve the lifted relaxation without cuts, and perform alternating minimization, then impose projection cuts only when the base relaxation is not sufficiently strong. 
}

\subsection{Scalability of Lifted Relaxation on Synthetic Data}\label{ssec:4.3}
{\color{black}
Next, motivated by the observation that relaxations without projection cuts may sometimes be preferable as $n$ increases, our third experiment benchmarks the scalability of our reduced lifted relaxation \eqref{eqn:lmrc_further_reduced} without projection cuts, as we increase $n$ with $m=50, k=2$. We do not impose any Frobenius regularization for this experiment, and thus the method of \citet{bertsimas2021new} is not applicable. We use the same data generation process as the previous experiment to generate $\bm{A}$. To ensure that our relaxation \eqref{eqn:lmrc_further_reduced} is non-trivial, we include a universal-query component in the observation design. Specifically, we set $q=\floor{\rho m}$, observe the first $q$ columns in every row, and then observe each entry in the remaining $m-q$ columns independently with probability $\tau=\frac{p-\sfrac{q}{m}}{1-\sfrac{q}{m}}$ in order that the expected number of observed entries is $pnm$. This combines the universal query motivation of \citet{goldberg2001eigentaste} with the missing completely at random sampling model of \citet{rubin1976inference}. 

We report the average lower bound over $10$ random seeds (divided by $nm$ so that quantities have the same scale as we vary $n$; left) and the average computational time (right) in Figure \ref{fig:gw_vary_n} for $p=0.4, \rho=0.3$, with Mosek's relative optimality tolerance set to $10^{-4}$. We also report the average objective value obtained by running alternating minimization from $10$ random initialization points and taking the best solution \citep{burer2003nonlinear, jain2013low} as a baseline. We observe that the lifted compact relaxation \eqref{eqn:lmrc_further_reduced} scales up to $n=350$. Further improvements in SDP algorithms and hardware may extend this range further. We observe similar scalability performance, albeit with a larger optimality gap, for $p=0.4, \rho=0.2$ in Figure \ref{fig:gw_vary_n_2}.
}
\begin{figure}[h!]
    \centering
\begin{subfigure}{0.45\textwidth}
        \includegraphics[width=\textwidth]{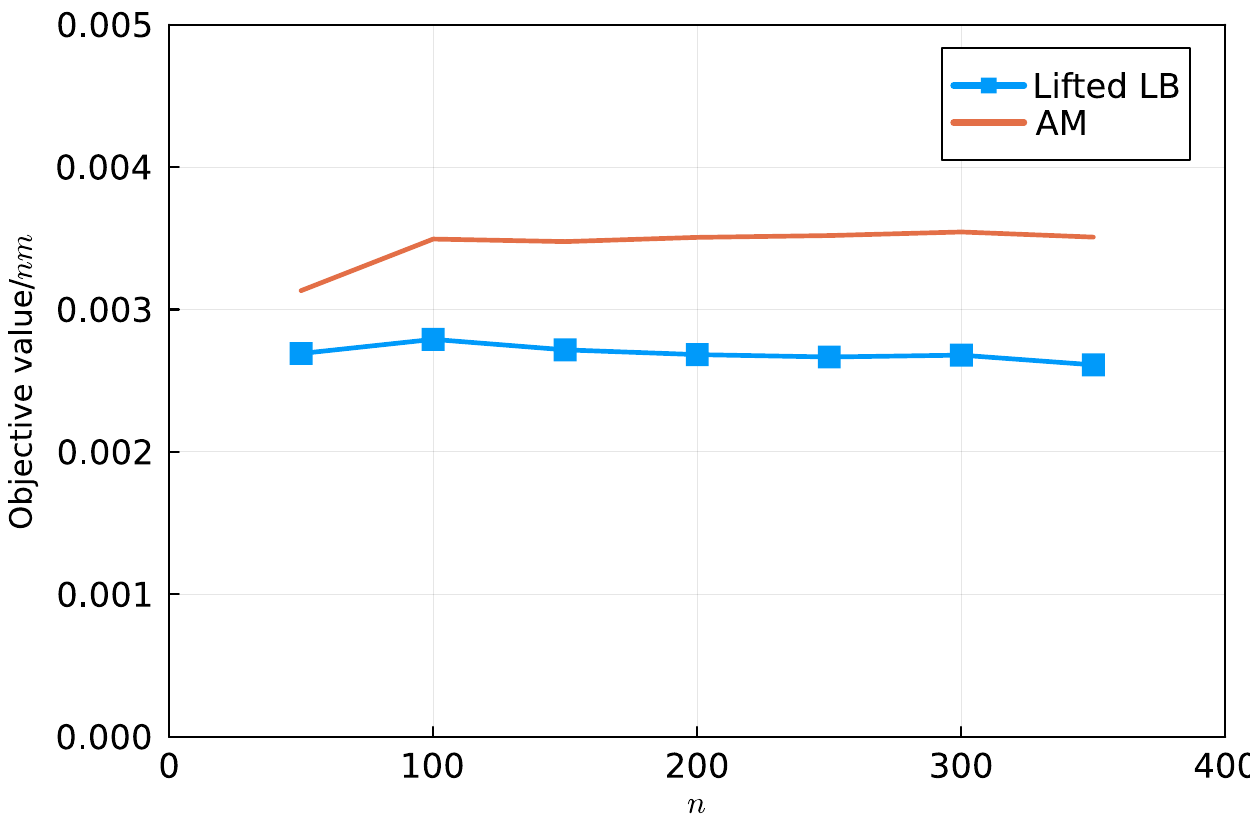}
        \caption{Objective values}
        \label{fig:bounds_0p5_objvals}
    \end{subfigure}
        \begin{subfigure}{0.45\textwidth}
        \includegraphics[width=\textwidth]{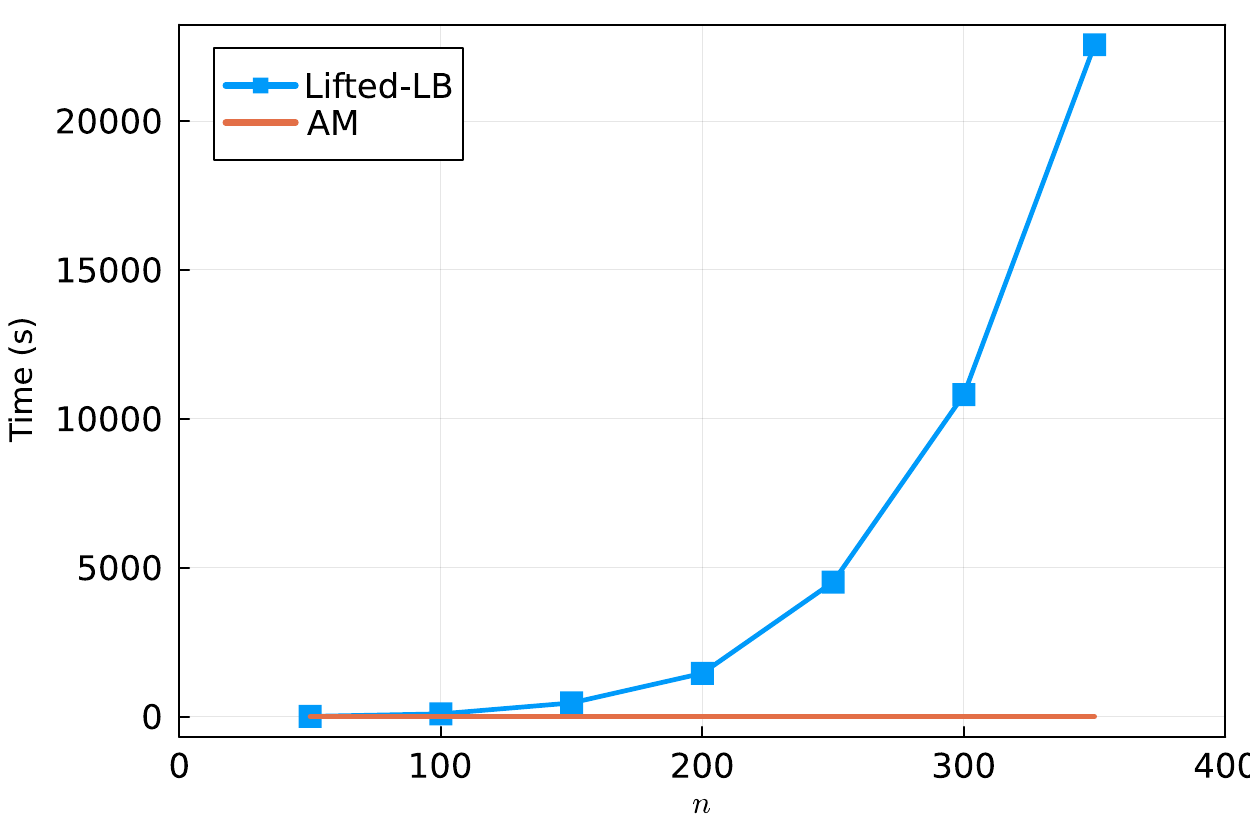}
        \caption{Runtimes}
        \label{fig:bounds_0p95_objvals}
    \end{subfigure}
    \caption{\color{black}Normalized objective values (left panel) and runtime for the compact lifted relaxation (right panel) as we vary $n$ with $m=50$ for our compact lifted relaxation \eqref{eqn:lmrc_further_reduced} with $p=0.4, \rho=0.3$.}
    \label{fig:gw_vary_n}
\end{figure}

\subsection{Performance on Real Data}\label{ssec:real.data}
{\color{black}
Our final experiment demonstrates the real-world applicability of our methods by evaluating our lifted relaxations on real missing-data matrices. We select five datasets from the UCI and RDatasets repos as indicated in Table \ref{tab:real-data-projection-cuts}. For each dataset, we drop the response variable and categorical variables, standardize numerical columns using observed entries, and preserve the native missingness pattern. We then construct ten $30 \times 10$ submatrices by selecting missing-affected numerical columns and rows enriched for missing entries, without masking any additional values. For $k \in \{2, 3\}$ and $\gamma/(nm)=10^5$, we solve the lifted relaxation \eqref{prob:lowrankrelaxation_reduced} with and without 1000 randomly sampled projection cuts \eqref{eqn:projcut_matcomp} of cardinality $k+1$. Table \ref{tab:real-data-projection-cuts} reports average runtimes and relative gaps against the same alternating-minimization upper bound used above. All SDP solver parameters are the same as in the previous experiment, except that we impose a relative optimality tolerance of $10^{-8}$, a memory limit of 128 GB of RAM, and a time limit of $8$ hours.

\begin{table}[t]
\centering
\caption{Real-data results for the lifted relaxation \eqref{prob:lowrankrelaxation_reduced} with and without $1000$ randomly selected projection cuts \eqref{eqn:projcut_matcomp}.
Each row reports averages over ten \(30\times 10\) submatrices constructed from the native missingness pattern of the indicated
dataset. We fix \(\gamma/(nm)=10^5\). ``Obs. frac.'' denotes the average fraction of observed entries in the selected 30×10 submatrices, while ``Miss. cols.'' denotes the average number of the 10 selected columns that contain at least one missing entry.}
\label{tab:real-data-projection-cuts}
\scriptsize
\setlength{\tabcolsep}{4pt}
\renewcommand{\arraystretch}{1.07}
\begin{tabular}{llccrrrrrrrr}
\toprule
& & & & \multicolumn{4}{c}{$k=2$} & \multicolumn{4}{c}{$k=3$} \\
\cmidrule(lr){5-8}\cmidrule(l){9-12}
& & & & \multicolumn{2}{c}{No cuts} & \multicolumn{2}{c}{Rand.} 
        & \multicolumn{2}{c}{No cuts} & \multicolumn{2}{c}{Rand.} \\
\cmidrule(lr){5-6}\cmidrule(lr){7-8}
\cmidrule(lr){9-10}\cmidrule(l){11-12}
Dataset & Source & Obs. & Miss. & Gap & T & Gap & T & Gap & T & Gap & T \\
        &        & frac. & cols. & (\%) & (s) & (\%) & (s) & (\%) & (s) & (\%) & (s) \\
\midrule
Arrhythmia  & UCI  & 0.891 & 2.7 & 26.00\% & 9.58 & 10.21 \% & 58.75 & 28.05\% & 8.97 & 14.84 \% & 104.43 \\
Cars93 & RDatasets             &  0.953 & 3 & 53.89\% & 9.24 & 8.32 \% & 54.74 &
64.93\% & 9.14 & 19.90 \% & 115.71 \\
MCAS & RDatasets   & 0.882 & 3 & 35.04\% & 9.18 & 11.7 \% & 58.64 &
46.41\% & 9.85 & 23.66 \% & 124.89 \\
ambientNOxCH     & RDatasets &  0.858 & 9.8 & 100\% & 9.33 & 24.48 \% & 67.66 &
100\% & 8.91 & 42.67 \% & 143 \\
movies & RDatasets & 0.806 & 2 & 0.06\% & 9.26 & 0.06 \% & 54.67 &
0.00\% & 9.26 & 0.00 \% & 100.4 \\\bottomrule
\end{tabular}
\end{table}

We observe that the strength of our base relaxation \eqref{prob:lowrankrelaxation_reduced} depends strongly on the dataset used, in the sense that it is tight ($0\%$ gap when $k=3$) with no projection cuts on the movies dataset, and $100\%$ with no projection cuts on the ambientNOxCH dataset. However, for datasets where the gap between our base relaxation \eqref{prob:lowrankrelaxation_reduced} and alternating minimization is not $0\%$, imposing $1000$ projection cuts is highly effective, reducing the optimality gap by more than half in most cases.
}

\section{Conclusion} 
This paper develops scalable semidefinite relaxations for rank-constrained quadratic optimization problems.  
{\color{black}Our starting point is a mixed-projection reformulation of the rank-constrained problems and a lifting}, obtained by introducing moment variables for the quadratic terms in $\bm{X}$ and the projection matrix $\bm{Y}$. While a direct lifted relaxation involves large moment matrices with $n^2 \times n^2$ semidefinite blocks, we show that most of these blocks are redundant and can be removed without altering the relaxation's optimal value. Our key insight is that, in a lifted relaxation, the effective size of a low-rank relaxation is governed by the subset of moments of the matrix queried by the objective, side constraints, and coupled with a projection matrix, rather than by the size of the matrix squared. 

{\color{black}Beyond this generic reduction, we derive a practical modeling workflow to exploit additional structure. For instance, for unregularized low-rank matrix completion, our relaxation can be expressed using the original matrix, a $m \times m$ semidefinite matrix, a $n \times n$ semidefinite matrix, and a single $(m+n) \times (m+n)$ semidefinite coupling constraint, eliminating any need to manipulate any matrix whose dimension is $O((n+m)^2)$. Collectively, our examples suggest that many large lifted relaxations can often be transformed into implementable SDPs.}

\theendnotes

\end{bibunit}

\newpage
\ECSwitch
\renewcommand\thealgorithm{EC.\arabic{algorithm}}
\setcounter{algorithm}{0} 
\ECHead{Supplementary Material}

\section{Supplemental Theoretical Results and Examples for Section \ref{sec:lowrankshorrelax}}

\subsection{Example Showing Valid Inequalities May Strengthen Lifted Relaxation}\label{ssec:ex1}

\begin{example}\label{example:notequivalent}
Consider a low-rank matrix completion problem of the form
\begin{align*}
 \min_{\bm{X} \in \mathbb{R}^{n \times m}} \quad & \frac{1}{2\gamma}\Vert \bm{X}\Vert_F^2+\frac{1}{2}\sum_{(i,j) \in \Omega} (X_{i,j}-A_{i,j})^2 \ \text{s.t.} \ \mathrm{rank}(\bm{X})\leq k.
\end{align*}

Let the problem data be $\gamma=100, k=2, n=7, m=5$, and suppose we are trying to impute the following matrix, where $*$ denotes a missing entry:
\begin{align*}
    \bm{A}=\begin{pmatrix}  
    -2 &  *   & -1 &  1 &  -1\\
  * & 4 & -4 & -5 & -4\\
  * & -3 & 1 & 4 & 3\\
  3 &  5 & -5 & -5 &  -1\\
  7 &  8  & -10  & -8  & 1\\
  3  & 1  & -2  & *  & 5\\
  7  & 7  & -13 & -8  & *\end{pmatrix}.
\end{align*}

{\color{black}
Then (using Mosek version 10.2 to solve all semidefinite relaxations):
\begin{itemize}
    \item The relaxation of \citet{bertsimas2021new} has an optimal objective value of $4.637$.
    \item The semidefinite relaxation \eqref{prob:lowrankrelaxation_reduced} has an objective value of 5.0875.
    \item The semidefinite relaxation \eqref{prob:lowrankrelaxation_reduced} augmented with all projection cuts \eqref{eqn:projcut_matcomp} for $\vert \mathcal{R}\vert=3$ has an objective value of 10.142.
    \item The method of \cite{burer2003nonlinear} finds a feasible solution with objective value 10.142. 
\end{itemize}
That is, the lifted relaxation itself has an optimality gap of around $50\%$, but projection cuts are sufficient to close this gap to $0$.
}
\end{example}

\subsection{Proof of Proposition \ref{prop:failureofshornaive}}\label{sec:failureshor1}{\color{black}

\proof{Proof of Proposition \ref{prop:failureofshornaive}}
Let $(\bm{X},\bm{U},\bm{V},\bm{W})$ be feasible for \eqref{eq:uv-shor-relaxation}. Then, by considering the appropriate principal submatrix it is immediate that
$$
\begin{pmatrix}
1 & \mathrm{vec}(\bm{X}^\top)^\top \\
\mathrm{vec}(\bm{X}^\top) & \bm{W}_{x,x}
\end{pmatrix}
\succeq 0.
$$
Moreover, the
pair $(\bm{X},\bm{W}_{x,x})$ is feasible for
\eqref{eq:rank-free-shor-relaxation} and attains the same objective value as in \eqref{eq:uv-shor-relaxation}. Therefore, the optimal value of \eqref{eq:uv-shor-relaxation} is at least that
of \eqref{eq:rank-free-shor-relaxation}.

Conversely, fix any feasible solution $(\bm{X},\bm{W}_{x,x})$ of
\eqref{eq:rank-free-shor-relaxation}. Set $\bm{U} = \bm{0}, \ \bm{V} = \bm{0}, \ \bm{W}_{x,u} = \bm{0}, \ \bm{W}_{x,v} = \bm{0}$. Next, partition $\bm{W}_{u,v}$ into $k \times k$ blocks
$\bm{W}_{u,v}^{(j,\ell)} \in \mathbb{R}^{n\times m}$ and define
$\bm{W}_{u,v}^{(1,1)} := \bm{X}, \ \bm{W}_{u,v}^{(j,\ell)} := 0 \ \text{for}\ (j,\ell) \neq (1,1).$ Next, set $\bm{W}_{u,u}^{(j,j)} := \frac{1}{n} \bm{I}_n \ \forall j \in [k], \ \bm{W}_{u,u}^{(j,l)} := \bm{0} \ \forall j, l \in [k]: j \neq l, \ \bm{W}_{v,v}^{(1,1)} := n \bm{X^\top X}$
with every other block of $\bm{W}_{v,v}$ equal to zero. By construction, $\sum_{j=1}^k \bm{W}_{u,v}^{(j,j)} = \bm{X}$
Moreover, the matrix
\[
\begin{pmatrix}
\bm{W}_{u,u} & \bm{W}_{u,v} \\
\bm{W}_{u,v}^\top & \bm{W}_{v,v}
\end{pmatrix}
\]
is block diagonal, with the only nontrivial coupled block being 
\[
\begin{pmatrix}
\frac{1}{n} \bm{I}_n & \bm{X} \\
\bm{X}^\top & n \bm{X}^\top \bm{X}
\end{pmatrix},
\]
which is clearly positive semidefinite. Therefore, the full moment matrix in \eqref{eq:uv-shor-relaxation} is block
diagonal:
\[
\begin{pmatrix}
1 & \mathrm{vec}(\bm{X}^\top)^\top & 0 & 0 \\
\mathrm{vec}(\bm{X}^\top) & \bm{W}_{x,x} & 0 & 0 \\
0 & 0 & \bm{W}_{u,u} & \bm{W}_{u,v} \\
0 & 0 & \bm{W}_{u,v}^\top & \bm{W}_{v,v}
\end{pmatrix}
\succeq 0,
\]
because its two diagonal blocks are positive semidefinite. Thus, the
constructed point is feasible for \eqref{eq:uv-shor-relaxation} and attains
the same objective value as $(\bm{X},\bm{W}_{x,x})$ in
\eqref{eq:rank-free-shor-relaxation}. It follows that the optimal value of \eqref{eq:uv-shor-relaxation} is at most that of \eqref{eq:rank-free-shor-relaxation}. Combining the two inequalities yields the result.
\hfill \Halmos \endproof

\begin{remark} The feasible solution constructed for the reverse inequality also satisfies the constraint $\mathrm{tr}(\bm{W}_{u,u}^{(j,l)})=\delta_{j,l}, \forall j,l \in [k]$.
\end{remark}
}

\subsection{Proof of Proposition \ref{prop:shorequiv2}}\label{ssec:shorequiv3}

{\color{black}
\proof{Proof of Proposition \ref{prop:shorequiv2}}
Let $z_{\mathrm{BM}}$ denote the optimal value of \eqref{eq:uv-shor-relaxation2} and $z_{\mathrm{proj}}$ denote the optimal value of \eqref{eqn:quadraticShorrelax_lowrank2}. We show that $z_{\mathrm{BM}}=z_{\mathrm{proj}}$. 

First, let $(\bm{X},\bm{Y},\bm{W}_{x,x})$ be feasible for \eqref{eqn:quadraticShorrelax_lowrank2}. We show that $(\bm{X},\bm{Y},\bm{W}_{x,x})$ generates a solution in \eqref{eq:uv-shor-relaxation2} with an equal or lower objective, and hence $z_{\mathrm{BM}}\leq z_{\mathrm{proj}}$. 

To show this, define $\bm{\Theta} := \sum_{i=1}^n \bm{W}_{x,x}^{(i,i)}, \ 
\tau := \frac{k-\operatorname{tr}(\bm{Y})}{kn} \ge 0.$ Moreover, set $\bm{U}=\bm{V}=\bm{0}$, $\bm{W}_{x,u}=\bm{W}_{x,v}=\bm{0}$,  all off-diagonal factor blocks to zero, and for each
$j \in [k]$ define
\begin{align*}
\bm{W}_{u,u}^{(j,j)} := \frac1k \bm{Y} + \tau \bm{I}_n,
\
\bm{W}_{u,v}^{(j,j)} := \frac1k \bm{X},
\
\bm{W}_{v,v}^{(j,j)} := \frac1k \bm{\Theta}.    
\end{align*}

Then, we have
\[
\operatorname{tr}\!\bigl(\bm{W}_{u,u}^{(j,j)}\bigr)=1,\
\sum_{j=1}^k \bm{W}_{u,v}^{(j,j)} = \bm{X},\
\sum_{j=1}^k \bm{W}_{v,v}^{(j,j)} = \bm{\Theta}.
\]
Moreover, for each $j$,
\[
\begin{pmatrix}
\bm{W}_{u,u}^{(j,j)} & \bm{W}_{u,v}^{(j,j)}\\
(\bm{W}_{u,v}^{(j,j)})^\top & \bm{W}_{v,v}^{(j,j)}
\end{pmatrix}
=
\frac1k
\begin{pmatrix}
\bm{Y} & \bm{X}\\
\bm{X}^\top & \bm{\Theta}
\end{pmatrix}
+
\begin{pmatrix}
\tau \bm{I}_n & \bm{0}\\
\bm{0} & \bm{0}
\end{pmatrix}
\succeq \bm{0}.
\]
Hence, the factor block is PSD, and together with
\[
\begin{pmatrix}
1 & \operatorname{vec}(\bm{X}^\top)^\top\\
\operatorname{vec}(\bm{X}^\top) & \bm{W}_{x,x}
\end{pmatrix}\succeq 0
\]
this gives a feasible point of \eqref{eq:uv-shor-relaxation2} with the same objective value and the same side constraints.

Conversely, let $(\bm{X},\bm{U},\bm{V},\bm{W})$ be feasible for \eqref{eq:uv-shor-relaxation2}, and define 
\begin{align*}
\bm{\widetilde Y} := \sum_{j=1}^k \bm{W}_{u,u}^{(j,j)},
\ \bm{\Theta} := \sum_{j=1}^k \bm{W}_{v,v}^{(j,j)}= \sum_{i=1}^n \bm{W}_{x,x}^{(i,i)}.
\end{align*}
Summing the principal $(u_j,v_j)$-blocks yields
\[
\begin{pmatrix}
\bm{\widetilde Y} & \bm{X}\\
\bm{X}^\top & \bm{\Theta}
\end{pmatrix}\succeq \bm{0},
\qquad
\operatorname{tr}(\bm{\widetilde Y})=k,
\]
while the principal $(1,\bm{X})$-block gives
\[
\bm{W}_{x,x}\succeq \operatorname{vec}(\bm{X}^\top)\operatorname{vec}(\bm{X}^\top)^\top.
\]
Summing the diagonal $m\times m$ blocks of the latter inequality gives
\[
\bm{\Theta} \succeq \bm{X^\top X}.
\]
Now set
\[
\bm{Y} := \bm{X}\bm{\Theta}^\dagger \bm{X}^\top.
\]
Since $\bm{\Theta} \succeq \bm{X^\top X}$, we have $\ker(\bm{\Theta})\subseteq \ker(\bm{X})$, and therefore
\[
\begin{pmatrix}
\bm{\Theta} & \bm{X}^\top\\
\bm{X} & \bm{Y}
\end{pmatrix}\succeq \bm{0}.
\]
Also, with $\bm{R}:=\bm{X}\bm{\Theta}^{\dagger/2}$,
\begin{align*}
\bm{R}^\top \bm{R}
= \bm{\Theta}^{\dagger/2}\bm{X}^\top \bm{X}\bm{\Theta}^{\dagger/2}
\preceq \bm{\Theta}^{\dagger/2}\bm{\Theta}\bm{\Theta}^{\dagger/2}
= P_{\operatorname{range}(\bm{\Theta})}
\preceq \bm{I},
\end{align*}
so $\bm{0} \preceq \bm{Y} = \bm{RR}^\top \preceq \bm{I}$. Finally, applying the generalized Schur complement to
\[
\begin{pmatrix}
\bm{\widetilde Y} & \bm{X}\\
\bm{X}^\top & \bm{\Theta}
\end{pmatrix}\succeq \bm{0}
\]
gives
\[
\bm{Y} = \bm{X}\bm{\Theta}^\dagger \bm{X}^\top \preceq \bm{\widetilde Y},
\]
hence
\[
\operatorname{tr}(\bm{Y})\le \operatorname{tr}(\bm{\widetilde Y})=k.
\]
Thus $(\bm{X},\bm{Y},\bm{W}_{x,x})$ is feasible for \eqref{eqn:quadraticShorrelax_lowrank2} with the same objective value
and the same side constraints. Therefore, $z_{\mathrm{proj}} \le z_{\mathrm{BM}}$. \Halmos
\endproof
}

\newpage

\section{Additional Examples Supporting Section \ref{sec:examplesofrelaxations}}\label{append.additionalexamples}
In this section, we first provide a proof of Proposition \ref{prop:reducedmatrixcompletionii} and then support our discussion in Section \ref{sec:examplesofrelaxations} by providing three additional families of low-rank problems that admit compact semidefinite relaxations using the lifting outlined in Section \ref{ssec:shorrelax}.

{\color{black}
\subsection{Proof of Proposition \ref{prop:reducedmatrixcompletionii}}\label{ssec:redmc2prop}
\proof{Proof of Proposition \ref{prop:reducedmatrixcompletionii}}
We fix $(\bm{X},\bm{Y})$ in \eqref{prob:lowrankrelaxation_reduced} and consider the partial minimization problem with respect to $\bm{S}^i$. We write
\begin{align*}
\bm{S}^i=\bm{X}_{i,.}\bm{X}_{i,.}^\top+\bm{U}^i,\ \bm{U}^i\succeq\bm{0}.
\end{align*}
Then, by the Schur complement lemma, the slack matrices $\bm{U}^i$ must satisfy
\begin{align*}
\sum_{i\in[n]} \bm{U}^i \succeq \bm{R}:=\bm{X}^\top \bm{Y}^\dagger \bm{X}-\bm{X}^\top \bm{X},
\end{align*}
whenever $\mathrm{range}(\bm{X})\subseteq \mathrm{range}(\bm{Y})$. Since $\bm{0}\preceq \bm{Y}\preceq \bm{I}$, we have
\(\bm{R}\succeq\bm{0}\). Minimizing for $\bm{S}^i$ is then equivalent to solving
\begin{align*}
\phi_H(\bm{R}):=
\min_{\substack{\bm{U}^i\succeq\bm{0}\\ \sum_i \bm{U}^i\succeq \bm{R}}}
\sum_{i\in[n]}\langle \bm{H}^i,\bm{U}^i\rangle
=
\max_{\substack{\bm{M}\succeq\bm{0}\\ \bm{M}\preceq \bm{H}^i\ \forall i\in[n]}}
\langle \bm{M},\bm{R}\rangle,
\end{align*}
where strong duality holds because the minimization problem satisfies Slater's constraint qualification. 

When the matrices $\bm{H}^i$ are diagonal and binary, we have
\begin{align*}
\{\bm{M}\succeq\bm{0}:\ \bm{M}\preceq \bm{H}^i\ \forall i\in[n]\} = \{\bm{M}:\ \bm{0}\preceq \bm{M}\preceq \bar{\bm{H}}\}.
\end{align*}
Indeed, let $S^i$ denote the support of the matrix $\bm{H}^i$. By construction, the support of $\bar{\bm{H}}$ is $\bar{S} := \cap_{i \in [n]} S^i$. If $\bm{M}\preceq \bm{H}^i$ for some $i \in [n]$, then the matrix $\bm{M}$ is supported only on the coordinates in $S^i$. Applying this reasoning for all $i \in [n]$, we can say that matrices in $\{\bm{M}\succeq\bm{0}:\ \bm{M}\preceq \bm{H}^i\ \forall i\in[n]\}$ are only supported on coordinates in $\bar{S}$.
Restricted to $\bar{S}$, the diagonal matrices $\bm{H}^i, i\in [n]$ and $\bar{\bm H}$ are all equal to the identity matrix, showing that both sets are equal.
Hence, because \(\bm{R}\succeq\bm{0}\), we must have 
$\phi_H(\bm{R})=\langle \bar{\bm{H}},\bm{R}\rangle$.
Therefore, for fixed $(\bm{X},\bm{Y})$, the optimized $\bm{S}^i$-dependent contribution in \eqref{prob:lowrankrelaxation_reduced} equals
\begin{align*}
\sum_{i\in[n]}\langle \bm{H}^i,\bm{X}_{i,.}\bm{X}_{i,.} ^\top\rangle
+
\langle \bar{\bm H},\bm{X}^\top \bm{Y}^\dagger \bm{X}-\bm{X}^\top \bm{X}\rangle,
\end{align*}
which can be rewritten as
\begin{align*}
\sum_{i\in[n]}\langle \bm{H}^i-\bar{\bm{H}},\bm{X}_{i,.} \bm{X}_{i,.}^\top\rangle+\langle \bar{\bm{H}},\bm{X}^\top \bm{Y}^\dagger \bm{X}\rangle. 
\end{align*}
Finally, minimizing $\langle \bar{\bm{H}}, \bm{\Theta}\rangle$ subject to the Schur complement constraint leads to the objective $\langle \bar{\bm{H}}, \bm{X}^\top \bm{Y}^\dagger \bm{X}\rangle$ since $\bar{\bm{H}} \succeq \bm{0}$. \Halmos
\endproof
}

\subsection{Sparse Plus Low-Rank Matrix Decomposition}\label{ssec:ex.slr}
In this section, we leverage the techniques proposed in this paper to obtain a strong convex relaxation for the sparse plus low-rank matrix decomposition problem:
\begin{align} \label{eq:slr}
\min_{\bm{X}\in \mathbb{R}^{n \times m}, \bm{F} \in \mathbb{R}^{n \times m}} \quad & \Vert \bm{A}-\bm{X}-\bm{F}\Vert_F^2 \quad \text{s.t.} \quad \mathrm{rank}(\bm{X}) \leq k, \ \Vert \bm{F}\Vert_0 \leq s,
\end{align}
where $\bm{A}$ is a data matrix which is to be approximately decomposed into a low-rank matrix $\bm{X}$ plus a sparse matrix $\bm{F}$. We remark that convex relaxations for this problem have been proposed for the special case where the objective is augmented with regularization by \cite{bertsimas2021sparse}. However, their matrix perspective relaxation gives a trivial lower bound of $0$ for \eqref{eq:slr}.

We now demonstrate that a non-trivial convex relaxation can be obtained by leveraging the ideas described in this paper. Specifically, we follow Section \ref{ssec:shorrelax} in introducing a projection matrix $\bm{Y} \in \mathcal{Y}^k_n$ to model the rank of $\bm{X}$, a binary matrix $\bm{Z}$ to model the sparsity of $\bm{F}$, and imposing the bilinear constraints $\bm{X}=\bm{Y}\bm{X}, \bm{F}=\bm{Z}\circ \bm{F}$. This gives the following reformulation of \eqref{eq:slr}:
\begin{align}
    \min_{\bm{Z} \in \{0, 1\}^{n \times m}, \bm{Y}\in \mathcal{Y}^k_n}\ \min_{\bm{X}\in \mathbb{R}^{n \times m}, \bm{F} \in \mathbb{R}^{n \times m}} \ & \Vert \bm{A}-\bm{X}-\bm{F}\Vert_F^2 & \text{s.t.} \ & \bm{X}=\bm{YX}, \bm{F}=\bm{Z}\circ \bm{F}, \mathrm{tr}(\bm{Y}) \leq k, \langle \bm{E}, \bm{Z}\rangle \leq s,
\end{align}
which is nonconvex through quadratic constraints and thus is amenable to a lifted relaxation. Following Proposition \ref{prop:convexrelax}, we vectorize all matrices, introduce a lifted matrix, and relax the constraints $\bm{Z} \in \{0, 1\}^{n \times m}$ and $\bm{Y}\in \mathcal{Y}^k_n$ to obtain the following semidefinite relaxation of \eqref{eq:slr}. 
\begin{align}
\label{eq:slr_shorrelax}    
\min_{\substack{
\bm{X},\bm{F}\in\mathbb{R}^{n\times m},\
\bm{Y}\in \mathrm{Conv}(\mathcal{Y}^k_n),\\
\bm{Z}\in[0,1]^{n\times m}:\ \langle\bm{E},\bm{Z}\rangle\le s,\\
\bm{W}_{x,x},\bm{W}_{y,y},\bm{W}_{f,f},\bm{W}_{z,z}\succeq 0,\\
\bm{W}_{x,y},\bm{W}_{x,f},\bm{W}_{f,z},\bm{W}_{x,z},\bm{W}_{y,f},\bm{W}_{y,z}}}
\quad
& \langle \bm{A},\bm{A}\rangle
-2\langle \bm{A},\bm{X}+\bm{F}\rangle
+\mathrm{tr}(\bm{W}_{x,x})
+\mathrm{tr}(\bm{W}_{f,f})
+2\,\mathrm{tr}(\bm{W}_{x,f})
\\
\text{s.t.}\quad
&
\begin{pmatrix}
1 & \mathrm{vec}(\bm{X}^\top)^\top & \mathrm{vec}(\bm{Y}^\top)^\top & \mathrm{vec}(\bm{F}^\top)^\top & \mathrm{vec}(\bm{Z}^\top)^\top\\
\mathrm{vec}(\bm{X}^\top) & \bm{W}_{x,x} & \bm{W}_{x,y} & \bm{W}_{x,f} & \bm{W}_{x,z}\\
\mathrm{vec}(\bm{Y}^\top) & \bm{W}_{x,y}^\top & \bm{W}_{y,y} & \bm{W}_{y,f} & \bm{W}_{y,z}\\
\mathrm{vec}(\bm{F}^\top) & \bm{W}_{x,f}^\top & \bm{W}_{y,f}^\top & \bm{W}_{f,f} & \bm{W}_{f,z}\\
\mathrm{vec}(\bm{Z}^\top)& \bm{W}_{x,z}^\top & \bm{W}_{y,z}^\top & \bm{W}_{f,z}^\top & \bm{W}_{z,z}
\end{pmatrix}\succeq \bm{0},
\nonumber\\[0.3em]
& \sum_{i\in[n]} \bm{W}_{y,y}^{(i,i)} = \bm{Y}, 
\qquad
\sum_{i\in[n]} \bm{W}_{x,y}^{(i,i)} = \bm{X}^\top,
\nonumber\\[0.3em]
& \mathrm{diag}(\bm{W}_{z,z}) = \mathrm{vec}(\bm{Z}^\top),
\qquad
\mathrm{diag}(\bm{W}_{f,z}) = \mathrm{vec}(\bm{F}^\top).
\nonumber
\end{align}

This relaxation is \textit{not} practically tractable. However, we observe that, for the support variables $\bm{Z}$, only the
\emph{diagonal} product information is needed, and that the Frobenius objective depends only on row-wise quadratic terms, so we can fix all off-diagonal moment blocks without loss of generality, using the same
off-diagonal elimination arguments as in Theorem \ref{thm:shorlowrankequiv} and {\color{black}Theorem \ref{prop:common-core-penalty} (note that Theorem \ref{prop:common-core-penalty} is not directly applicable due to the sparse matrix)}. This gives rise to the following application of Theorem \ref{thm:shorlowrankequiv} to Problem \eqref{eq:slr_shorrelax}, {\color{black}where $\bm{E}$ denotes a matrix of all ones}:
\begin{proposition}\label{prop:slr}
    Problem \eqref{eq:slr_shorrelax} attains the same value as the following compact relaxation:
    \begin{align}
        \label{eq:slr_sdprelax_compact}
\min_{\substack{
\bm{X},\bm{F}, \bm{T}, \bm{V}\in\mathbb{R}^{n\times m}, \
\bm{S}^i\in\mathcal{S}_+^m,\\
\bm{Y}\in\mathrm{Conv}(\mathcal{Y}^k_n),\
\bm{Z}\in[0, 1]^{n \times m}: \langle \bm{E}, \bm{Z}\rangle \leq s}}
\quad
& \langle \bm{A},\bm{A}\rangle
-2\langle \bm{A},\bm{X}+\bm{F}\rangle
+\sum_{i\in[n]} \mathrm{tr}(\bm{S}^i)
+\langle \bm{E}, \bm{T}\rangle
+2\langle \bm{E}, \bm{V}\rangle
\\
\text{s.t.}\quad & \begin{pmatrix}
\sum_{i\in[n]} \bm{S}^i & \bm{X}^\top\\
\bm{X} & \bm{Y}
\end{pmatrix}\succeq \bm{0}, \ \bm{S}^{i}\succeq \bm{X}_{i,.}\bm{X}_{i,.}^\top \ \forall i \in [n],\nonumber\\
& \begin{pmatrix}
T_{ij} & F_{ij}\\
F_{ij} & Z_{ij}
\end{pmatrix}\succeq 0, \ \begin{pmatrix}
1 & X_{ij} & F_{ij}\\
X_{ij} & S^i_{jj} & V_{ij}\\
F_{ij} & V_{ij} & T_{ij}
\end{pmatrix}\succeq 0,\forall (i,j)\in[n]\times[m].\nonumber
\end{align}
\end{proposition}
\begin{remark}
    If only entries $\Omega\subseteq[n]\times[m]$ are observed in $\bm{A}$, we can replace the original objective function by $\Vert\mathcal{P}_\Omega(\bm{A}-\bm{X}-\bm{F})\Vert_F^2$ as in Section \ref{ssec:ex.mc}. The convex relaxation \eqref{eq:slr_sdprelax_compact} can be applied to this modified problem after modifying \eqref{eq:slr_sdprelax_compact}'s objective function 
    with row masks $\bm{H}^i$ as in Section \ref{ssec:ex.mc}. Further, these row masks allow us to make the relaxation more tractable by omitting the variables $V_{i,j}, T_{i,j}, Z_{i,j}$ for $(i,j) \notin \Omega$ and corresponding $2\times 2$ and $3 \times 3$ minors.
\end{remark}

\proof{Proof of Proposition \ref{prop:slr}}
We show that for any feasible solution to \eqref{eq:slr_shorrelax}, there exists a feasible solution to \eqref{eq:slr_sdprelax_compact} with the same objective value and vice versa.

To prove the result, we require additional notation. Namely, since \(\mathrm{vec}(\bm{X}^\top)\) stacks the rows of \(\bm{X}\), we index its entries by $\ell(i,j) := (i-1)m + j,\ (i,j)\in[n]\times[m],$
so that \(\mathrm{vec}(\bm{X}^\top)_{\ell(i,j)} = X_{ij}\).

\textbf{Problem \eqref{eq:slr_sdprelax_compact} is a relaxation of Problem \eqref{eq:slr_shorrelax}:} Fix a solution to \eqref{eq:slr_shorrelax} and define
\[
\bm{S}^i := \bm{W}_{x,x}^{(i,i)}\in\mathcal{S}_+^m,\quad i\in[n],\qquad
T_{ij} := (\bm{W}_{f,f})_{\ell(i,j),\ell(i,j)},\qquad
V_{ij} := (\bm{W}_{x,f})_{\ell(i,j),\ell(i,j)}.
\]
We claim that \((\bm{X},\bm{F},\bm{Y},\bm{Z},\{\bm{S}^i\}_{i\in[n]},\bm{T},\bm{V})\) is feasible for \eqref{eq:slr_sdprelax_compact} and attains the same objective value. First, the fact that the constraints 
\begin{align}\nonumber
\begin{pmatrix}
\sum_{i\in[n]} \bm{W}_{x,x}^{(i,i)} & \sum_{i\in[n]}\bm{W}_{x,y}^{(i,i)}\\[0.2em]
\left(\sum_{i\in[n]}\bm{W}_{x,y}^{(i,i)}\right)^\top & \sum_{i\in[n]}\bm{W}_{y,y}^{(i,i)}
\end{pmatrix}\succeq 0, \begin{pmatrix}
\sum_{i\in[n]} \bm{S}^i & \bm{X}^\top\\
\bm{X} & \bm{Y}
\end{pmatrix}\succeq 0, \bm{S}^i=\bm{W}_{x,x}^{(i,i)}\succeq \bm{X}_{i,.}\bm{X}_{i,.}^\top
\end{align}
hold follows identically to the proof of Theorem \ref{thm:shorlowrankequiv}. 

Second, for any \((i,j)\), consider the principal submatrix of the lifted PSD constraint indexed by
\((F_{ij},Z_{ij})\), i.e., the entries \(\ell(i,j)\) in \(\mathrm{vec}(\bm{F}^\top)\) and \(\mathrm{vec}(\bm{Z}^\top)\). This gives
\begin{align*}
\begin{pmatrix}
(\bm{W}_{f,f})_{\ell(i,j),\ell(i,j)} & (\bm{W}_{f,z})_{\ell(i,j),\ell(i,j)}\\
(\bm{W}_{f,z})_{\ell(i,j),\ell(i,j)} & (\bm{W}_{z,z})_{\ell(i,j),\ell(i,j)}
\end{pmatrix}\succeq 0.
\end{align*}
By the diagonal constraints in \eqref{eq:slr_shorrelax},
\((\bm{W}_{f,z})_{\ell(i,j),\ell(i,j)} = F_{ij}\) and \((\bm{W}_{z,z})_{\ell(i,j),\ell(i,j)} = Z_{ij}\).
Hence
\[
\begin{pmatrix}
T_{ij} & F_{ij}\\
F_{ij} & Z_{ij}
\end{pmatrix}\succeq 0,
\]
which is exactly the second set of LMIs in \eqref{eq:slr_sdprelax_compact}. 

Third, for any \((i,j)\), the principal submatrix of the lifted PSD constraint indexed by \((1,X_{ij},F_{ij})\),
i.e., \((1,\mathrm{vec}(\bm{X}^\top)_{\ell(i,j)},\mathrm{vec}(\bm{F}^\top)_{\ell(i,j)})\), yields
\[
\begin{pmatrix}
1 & X_{ij} & F_{ij}\\
X_{ij} & (\bm{W}_{x,x})_{\ell(i,j),\ell(i,j)} & (\bm{W}_{x,f})_{\ell(i,j),\ell(i,j)}\\
F_{ij} & (\bm{W}_{x,f})_{\ell(i,j),\ell(i,j)} & (\bm{W}_{f,f})_{\ell(i,j),\ell(i,j)}
\end{pmatrix}\succeq 0.
\]
Using \((\bm{W}_{x,x})_{\ell(i,j),\ell(i,j)} = (\bm{S}^i)_{jj}\), the definitions of \(T_{ij}\) and \(V_{ij}\),
we obtain
\[
\begin{pmatrix}
1 & X_{ij} & F_{ij}\\
X_{ij} & S^i_{jj} & V_{ij}\\
F_{ij} & V_{ij} & T_{ij}
\end{pmatrix}\succeq 0,
\]
which proves \((\bm{X},\bm{F},\bm{Y},\bm{Z},\{\bm{S}^i\}_{i\in[n]},\bm{T},\bm{V})\) is indeed feasible for \eqref{eq:slr_sdprelax_compact} with the same objective.

\textbf{Problem \eqref{eq:slr_shorrelax} is a relaxation of \eqref{eq:slr_sdprelax_compact}:} Fix a feasible solution \((\bm{X},\bm{F},\bm{Y},\bm{Z},\{\bm{S}^i\},\bm{T},\bm{V})\)
to \eqref{eq:slr_sdprelax_compact}, with a view to construct a solution to \eqref{eq:slr_shorrelax} with the same objective value. 

First, define \(\bm{W}_{x,x}\in\mathcal{S}^{nm}_+\) blockwise via
$\bm{W}_{x,x}^{(i,i)} := \bm{S}^i,\ \bm{W}_{x,x}^{(i,k)} := \bm{X}_{i,.}\bm{X}_{k,.}^\top\ \ (i\neq k).$
Then, $\bm{W}_{x,x}$ is positive semidefinite by construction and \(\mathrm{tr}(\bm{W}_{x,x})=\sum_i \mathrm{tr}(\bm{S}^i)\). \

Second, let \(\bm{\Theta}:=\sum_{i\in[n]}\bm{S}^i\).
From \eqref{eq:slr_sdprelax_compact} 
we have \(\bm{Y}\succeq \bm{X}\bm{\Theta}^\dagger \bm{X}^\top\).
Since \(\bm{Y}\) does not appear in the objective, we may replace \(\bm{Y}\) by
\(\widetilde{\bm{Y}}:=\bm{X}\bm{\Theta}^\dagger \bm{X}^\top\preceq \bm{Y}\), which preserves feasibility
in \(\mathrm{Conv}(\mathcal{Y}^k_n)\). Therefore, define \(\bm{U}:=\bm{\Theta}^\dagger \bm{X}^\top\in\mathbb{R}^{m\times n}\) and set the diagonal blocks $\bm{W}_{x,y}^{(i,i)} := \bm{S}^i \bm{U},\ \bm{W}_{y,y}^{(i,i)} := \bm{U}^\top \bm{S}^i \bm{U}.$
Then \(\sum_i \bm{W}_{x,y}^{(i,i)} = \bm{X}^\top\) and \(\sum_i \bm{W}_{y,y}^{(i,i)} = \widetilde{\bm{Y}}\).
By Theorem \ref{thm:shorlowrankequiv}'s construction, we can complete the remaining (off-diagonal) blocks of
\(\bm{W}_{x,y}\) and \(\bm{W}_{y,y}\) so the lifted PSD constraint on \((1,\mathrm{vec}(\bm{X}^\top),\mathrm{vec}(\bm{Y}^\top))\) holds. 

{\color{black}
Third, let 
\[
\bm{M}_{x,y}:=
\begin{pmatrix}
1 & \mathrm{vec}(\bm{X}^\top)^\top & \mathrm{vec}(\tilde{\bm{Y}}^\top)^\top\\
\mathrm{vec}(\bm{X}^\top) & \bm{W}_{x,x} & \bm{W}_{x,y}\\
\mathrm{vec}(\tilde{\bm{Y}}^\top) & \bm{W}_{x,y}^\top & \bm{W}_{y,y}
\end{pmatrix}\succeq \bm{0}
\]
denote the positive semidefinite matrix constructed above, and let
\[
g_0,\ \{g_{\ell(i,j)}\}_{(i,j)\in[n]\times[m]},\ \{h_q\}_{q=1}^{n^2}
\]
be vectors whose Gram matrix equals $\bm{M}_{x,y}$. In particular,
\[
\langle g_0,g_{\ell(i,j)}\rangle=X_{ij},\qquad \|g_{\ell(i,j)}\|^2=(S^i)_{jj}.
\]
For each $(i,j)$, write
\[
r_{ij}:=g_{\ell(i,j)}-X_{ij}g_0,
\]
so that $r_{ij}\perp g_0$ and $\|r_{ij}\|^2=(S^i)_{jj}-X_{ij}^2$.

Next, fix $(i,j)\in[n]\times[m]$. Since
\[
\begin{pmatrix}
1 & X_{ij} & F_{ij}\\
X_{ij} & (S^i)_{jj} & V_{ij}\\
F_{ij} & V_{ij} & T_{ij}
\end{pmatrix}\succeq 0,
\]
there exists a vector $q_{ij}$ such that
\[
\langle g_0,q_{ij}\rangle=F_{ij},\qquad 
\langle g_{\ell(i,j)},q_{ij}\rangle=V_{ij},\qquad 
\|q_{ij}\|^2=T_{ij}.
\]
For example, if $\|r_{ij}\|^2>0$, one may take
\[
q_{ij}=F_{ij}g_0+\alpha_{ij}r_{ij}+\beta_{ij}e_{ij},
\qquad
\alpha_{ij}:=\frac{V_{ij}-X_{ij}F_{ij}}{\|r_{ij}\|^2},
\]
where $e_{ij}$ is a unit vector orthogonal to all previously defined vectors, and
\[
\beta_{ij}:=\sqrt{T_{ij}-F_{ij}^2-\alpha_{ij}^2\|r_{ij}\|^2},
\]
which is well-defined by the above $3\times 3$ LMI. If $\|r_{ij}\|^2=0$, then the same LMI implies
$V_{ij}=X_{ij}F_{ij}$, and we may take
\[
q_{ij}=F_{ij}g_0+\sqrt{T_{ij}-F_{ij}^2}\,e_{ij}.
\]

Now define
\[
u_{ij}:=q_{ij}-F_{ij}g_0,
\]
so that $u_{ij}\perp g_0$ and $\|u_{ij}\|^2=T_{ij}-F_{ij}^2$. Since
\[
\begin{pmatrix}
T_{ij} & F_{ij}\\
F_{ij} & Z_{ij}
\end{pmatrix}\succeq 0
\qquad\text{and}\qquad 0\le Z_{ij}\le 1,
\]
there exists a vector $z_{ij}$ such that
\[
\langle g_0,z_{ij}\rangle=Z_{ij},\qquad
\langle q_{ij},z_{ij}\rangle=F_{ij},\qquad
\|z_{ij}\|^2=Z_{ij}.
\]
Indeed, if $\|u_{ij}\|^2>0$, we may take
\[
z_{ij}=Z_{ij}g_0+a_{ij}u_{ij}+b_{ij}\tilde e_{ij},
\qquad
a_{ij}:=\frac{F_{ij}(1-Z_{ij})}{\|u_{ij}\|^2},
\]
where $\tilde e_{ij}$ is a unit vector orthogonal to all previously defined vectors, and
\[
b_{ij}:=\sqrt{Z_{ij}-Z_{ij}^2-a_{ij}^2\|u_{ij}\|^2},
\]
which is well-defined because
\[
Z_{ij}-Z_{ij}^2-a_{ij}^2\|u_{ij}\|^2
=
\frac{(1-Z_{ij})(Z_{ij}T_{ij}-F_{ij}^2)}{T_{ij}-F_{ij}^2}\ge 0.
\]
If $\|u_{ij}\|^2=0$, then $q_{ij}=F_{ij}g_0$, and we may take
\[
z_{ij}=Z_{ij}g_0+\sqrt{Z_{ij}-Z_{ij}^2}\,\tilde e_{ij},
\]
noting that in this case $\langle q_{ij},z_{ij}\rangle=F_{ij}$ follows from
$F_{ij}^2=T_{ij}\le T_{ij}Z_{ij}=F_{ij}^2Z_{ij}$.

Finally, define the remaining lifted blocks so that the full lifted matrix in \eqref{eq:slr_shorrelax}
is the Gram matrix of
\[
\{g_0\}\cup \{g_{\ell(i,j)}\}_{(i,j)}\cup \{h_q\}_{q=1}^{n^2}
\cup \{q_{ij}\}_{(i,j)}\cup \{z_{ij}\}_{(i,j)}.
\]
This matrix is positive semidefinite by construction, and satisfies
\[
\operatorname{diag}(W_{f,f})=\vec(T^\top),\qquad
\operatorname{diag}(W_{f,z})=\vec(F^\top),\qquad
\operatorname{diag}(W_{z,z})=\vec(Z^\top),
\]
and also satisfies $\operatorname{tr}(W_{x,f})=\langle E,V\rangle$. Since the blocks
$W_{x,z},W_{y,f},W_{y,z}$ are unconstrained in \eqref{eq:slr_shorrelax}, the Gram values assigned to them are admissible.
Therefore, the constructed solution is feasible for \eqref{eq:slr_shorrelax} and has the same objective value as in \eqref{eq:slr_sdprelax_compact}.}
\Halmos \endproof

\subsection{Example: Basis Pursuit}\label{ssec:basispursuit}

Given a sample $\{A_{i,j}, (i,j)\in \Omega \subseteq [n] \times [m]\}$ of an \textit{exactly} low-rank matrix $\bm{A} \in \mathbb{R}^{n \times m}$, the goal of the low-rank basis pursuit problem is to recover the lowest rank matrix $\bm{X}$ that exactly matches all observed entries of $\bm{A}$ \citep{candes2009exact}. This problem admits the formulation: 
\begin{align}\label{prob:lrmbp_orig}
    \min_{\bm{Y} \in \mathcal{Y}_n }\min_{\bm{X}\in \mathbb{R}^{n \times m}} \quad &  \mathrm{tr}(\bm{Y}) \ \text{s.t.} \quad \mathcal{P}(\bm{A})=\mathcal{P}(\bm{X}), \bm{X}=\bm{Y}\bm{X},
\end{align}
where {\color{black}$\mathcal{Y}_n := \mathcal{Y}_n^n$ denote the set of $n \times n$ orthogonal projection matrices of any rank, and}
$\mathcal{P}(\bm{A})$ denotes a linear map that masks the hidden entries of $\bm{A}, \bm{X}$ such that $\mathcal{P}(\bm{A})_{i,j}=A_{i,j}$ if $(i,j) \in \Omega$ and $0$ otherwise. Following Theorem \ref{thm:shorlowrankequiv} and {\color{black}pairwise multiplying} the constraints $A_{i,j} - X_{i,j} = 0, \forall (i,j) \in \Omega$  leads to the following relaxation 
\begin{equation} \label{prob:basispursuit_shorrelax}
\begin{aligned}
    \min_{\bm{Y} \in \mathrm{Conv}(\mathcal{Y}_n) } \: \min_{\bm{X}\in \mathbb{R}^{n \times m}, \bm{W} \in \mathcal{S}^{nm}_+} & \mathrm{tr}(\bm{Y}) \\
    \text{s.t.} \quad
    & A_{i,j} A_{k,\ell} -A_{k,\ell} X_{i,j} -A_{i,j} X_{k,\ell} + (\bm{W}^{(i,k)})_{j,\ell} = 0, \forall (i,j), (k,\ell) \in \Omega \times \Omega \\
    & A_{i,j}=X_{i,j}, \forall (i,j) \in \Omega\\
    & \bm{W}\succeq \mathrm{vec}(\bm{X}^\top)\mathrm{vec}(\bm{X}^\top)^\top, \ \begin{pmatrix} \sum_{i \in [n]}\bm{W}^{(i,i)} & \bm{X}^\top \\ \bm{X} & \bm{Y}\end{pmatrix} \succeq \bm{0},
\end{aligned}
\end{equation}


Similarly to the low-rank matrix completion case, the structure of the compact lifted relaxation means that 
{\color{black}if we set $\bm{W}_{x,x}^{(i,j)}=\bm{X}_{i,.}\bm{X}_{j,.}^\top$ for $i \neq j$ then $X_{i,j}=A_{i,j}$ and $X_{k,l}=A_{k,l}$ for observed entries $(i,j), (k,l) \in \Omega$. Hence, any constraints involving the off-diagonal blocks of $\bm{W}_{x,x}$ are satisfied by the construction in the proof of Theorem \ref{prop:common-core-penalty}, and we have the following corollary to it}:
\begin{corollary}\label{corr:basispursuit1}
Problem \eqref{prob:basispursuit_shorrelax} attains the same objective value as
\begin{equation} \label{prob:basispursuit_compact}
\begin{aligned}
    \min_{\bm{Y} \in \mathrm{Conv}(\mathcal{Y}_n) } \: \min_{\bm{X}\in \mathbb{R}^{n \times m},\ \bm{S}^i \in \mathcal{S}^{m}_+, i\in[n]} & \mathrm{tr}(\bm{Y}) \\
    \text{s.t.} \quad
    & A_{i,j} A_{i,\ell} -A_{i,\ell} X_{i,j} -A_{i,j} X_{i,\ell} + (\bm{S}^{i})_{j,\ell} = 0, \forall (i,j), (i,\ell) \in \Omega \times \Omega \\
    & A_{i,j}=X_{i,j}, \forall (i,j) \in \Omega\\
    & \bm{S}^i\succeq \bm{X}_{i,\cdot}  \bm{X}_{i,\cdot}^\top, \ \begin{pmatrix} \sum_{i \in [n]}\bm{S}^{i} & \bm{X}^\top \\ \bm{X} & \bm{Y}\end{pmatrix} \succeq \bm{0},
\end{aligned}
\end{equation}
where $\bm{X}_{i,\cdot}$ denotes a column vector containing the $i$th row of $\bm{X}$.
\end{corollary}
\begin{remark}
    The number of linear equality constraints in \eqref{prob:basispursuit_compact} grows quadratically in $\vert \Omega\vert$, thus, in practice, one may subsample or aggregate these constraints to improve the tractability of the semidefinite relaxation.
\end{remark}

\subsection{Example: Low-Rank Factor Analysis}
\label{sec:ex.fa}
In this section, we apply the lifted relaxation technique developed in Section \ref{ssec:shorrelax}
to derive a strong convex relaxation for low-rank factor analysis as explored in \cite{bertsimas2017certifiably}:
\begin{align}
\label{eq:fa}
\min_{\bm{X}\in\mathcal{S}^{n}_+,\ \bm{D}\in\mathcal{D}} \quad
& \|\bm{A}-\bm{X}-\bm{D}\|_F^2
\quad \text{s.t.}\quad \mathrm{rank}(\bm{X})\le k,
\end{align}
where $\bm{A}\in\mathcal{S}^{n}_{++}$ is a covariance matrix which is to be approximately decomposed into a positive semidefinite low-rank matrix $\bm{X}$ plus a diagonal matrix $\bm{D}$, which is contained in a convex set $\mathcal{D}$ (e.g., $\mathcal{D}=\{\mathrm{Diag}(\bm{d}): \bm{d}\ge \bm{0}\}$. Note that \eqref{eq:fa} is nonconvex \emph{only} through the rank constraint.

As in Section~\ref{ssec:shorrelax}, we introduce a projection matrix
$\bm{Y}\in\mathcal{Y}^k_n$ and impose $\bm{X}=\bm{Y}\bm{X}$ to model the rank constraint. This yields
the equivalent mixed-projection formulation
\begin{align}
\label{eq:fa_mixed}
\min_{\bm{Y}\in\mathcal{Y}^k_n}\ \min_{\bm{X}\in\mathcal{S}^{n}_+,\ \bm{D}\in\mathcal{D}} \quad
& \|\bm{A}-\bm{X}-\bm{D}\|_F^2
\quad \text{s.t.}\quad \bm{X}=\bm{Y}\bm{X}.
\end{align}

As in previous examples, we now develop a semidefinite relaxation by vectorizing and lifting. This yields the following relaxation:
\begin{align}
\label{eq:fa_shorrelax}
\min_{\substack{
\bm{X}\in\mathcal{S}^{n}_+,\ \bm{d}\in\mathbb{R}^n,\ \bm{Y}\in \mathrm{Conv}(\mathcal{Y}^k_n),\\
\mathrm{Diag}(\bm{d})\in\mathcal{D},\ \bm{W}_{x,x},\bm{W}_{y,y}\in\mathcal{S}_+^{n^2},\\ \bm{W}_{d,d}\in\mathcal{S}_+^{n},\ \bm{W}_{x,y}\in\mathbb{R}^{n^2\times n^2},\\ \bm{W}_{x,d}\in\mathbb{R}^{n^2\times n},\ \bm{W}_{y,d}\in\mathbb{R}^{n^2\times n}}}
\quad
& \langle \bm{A},\bm{A}\rangle
-2\langle \bm{A},\bm{X}\rangle
-2\,\mathrm{diag}(\bm{A})^\top \bm{d}
+\mathrm{tr}(\bm{W}_{x,x})
+\mathrm{tr}(\bm{W}_{d,d})
+2\langle \bm{P},\bm{W}_{x,d}\rangle
\\
\text{s.t.}\quad
&
\begin{pmatrix}
1 & \mathrm{vec}(\bm{X}^\top)^\top & \mathrm{vec}(\bm{Y}^\top)^\top & \bm{d}^\top\\
\mathrm{vec}(\bm{X}^\top) & \bm{W}_{x,x} & \bm{W}_{x,y} & \bm{W}_{x,d}\\
\mathrm{vec}(\bm{Y}^\top) & \bm{W}_{x,y}^\top & \bm{W}_{y,y} & \bm{W}_{y,d}\\
\bm{d} & \bm{W}_{x,d}^\top & \bm{W}_{y,d}^\top & \bm{W}_{d,d}
\end{pmatrix}\succeq \bm{0},
\nonumber\\[0.3em]
& \sum_{i\in[n]} \bm{W}_{y,y}^{(i,i)} = \bm{Y},
\qquad
\sum_{i\in[n]} \bm{W}_{x,y}^{(i,i)} = \bm{X}^\top,
\nonumber
\end{align}
where, to simplify notation, we define the diagonal selector $\bm{P}\in\mathbb{R}^{n^2\times n}$
by $\mathrm{vec}(\mathrm{Diag}(\bm{d}))=\bm{P}\bm{d}$.

We now eliminate most lifted blocks without loss of generality, analogously to Theorem~\ref{thm:shorlowrankequiv} and Theorem~\ref{prop:common-core-penalty} {(\color{black}note that Theorem \ref{prop:common-core-penalty} is not directly applicable due to the diagonal matrix)}. Formally, we have the following result:

\begin{proposition}\label{prop:fa}
Problem \eqref{eq:fa_shorrelax} attains the same value as the following compact relaxation:
\begin{align}
\label{eq:fa_sdprelax_compact}
\min_{\substack{
\bm{X}\in\mathcal{S}^{n}_+,\ \bm{d}\in\mathbb{R}^n,\ \bm{t},\bm{v}\in\mathbb{R}^n,\\
\bm{S}^i\in\mathcal{S}_+^n,\ \bm{Y}\in\mathrm{Conv}(\mathcal{Y}^k_n),\ \mathrm{Diag}(\bm{d})\in\mathcal{D}}}
\quad
& \langle \bm{A},\bm{A}\rangle
-2\langle \bm{A},\bm{X}\rangle
-2\,\mathrm{diag}(\bm{A})^\top \bm{d}
+\sum_{i\in[n]} \mathrm{tr}(\bm{S}^i)
+\bm{e}^\top \bm{t}
+2\,\bm{e}^\top \bm{v}
\\
\text{s.t.}\quad
&
\begin{pmatrix}
\sum_{i\in[n]} \bm{S}^i & \bm{X}^\top\\
\bm{X} & \bm{Y}
\end{pmatrix}\succeq \bm{0},
\qquad
\bm{S}^{i}\succeq \bm{X}_{i,.}\bm{X}_{i,.}^\top\ \ \forall i \in [n],
\nonumber\\
&
\begin{pmatrix}
1 & X_{ii} & d_i\\
X_{ii} & S^i_{ii} & v_i\\
d_i & v_i & t_i
\end{pmatrix}\succeq \bm{0},
\qquad \forall i\in[n].
\nonumber
\end{align}
\end{proposition}


\begin{remark}\textbf{Comparison with \cite{bertsimas2017certifiably}:} Relaxation \eqref{eq:fa_sdprelax_compact} is a strong \emph{one-shot} semidefinite lower bound for the least-squares objective in \eqref{eq:fa}, without requiring additional regularization or bounding-box refinements, and which also applies when $\bm{A}$ is indefinite (although we lose the statistical interpretation in this case). By contrast, \cite{bertsimas2017certifiably} develop semidefinite bounds tailored to classical factor analysis
structure through a spectral reformulation that strongly leverages an additional constraint $\bm{A}-\bm{D}\succeq \bm{0}$, and their approach does not apply without this constraint. Thus, the approach derived here is more general, although it may be weaker after the McCormick branch-and-bound procedure in \cite{bertsimas2017certifiably} is applied. 
\end{remark}

\proof{Proof of Proposition \ref{prop:fa}}
We show that for any feasible solution to \eqref{eq:fa_shorrelax}, there exists a feasible solution to
\eqref{eq:fa_sdprelax_compact} with the same objective value and vice versa.

To prove the result, we require additional notation. Since $\mathrm{vec}(\bm{X}^\top)$ stacks the rows of $\bm{X}$,
we index its entries by $\ell(i,j):=(i-1)n+j,\ (i,j)\in[n]\times[n],$ so that $\mathrm{vec}(\bm{X}^\top)_{\ell(i,j)}=X_{ij}$.

\textbf{Problem \eqref{eq:fa_sdprelax_compact} is a relaxation of Problem \eqref{eq:fa_shorrelax}:}
Fix a solution to \eqref{eq:fa_shorrelax} and define
\[
\bm{S}^i := \bm{W}_{x,x}^{(i,i)}\in\mathcal{S}_+^n,\quad i\in[n],\qquad
t_i := (\bm{W}_{d,d})_{ii},\qquad
v_i := (\bm{W}_{x,d})_{\ell(i,i),\,i},\quad i\in[n].
\]
We claim that $(\bm{X},\bm{d},\bm{Y},\{\bm{S}^i\}_{i\in[n]},\bm{t},\bm{v})$ is feasible for \eqref{eq:fa_sdprelax_compact}
and attains the same objective value. First, the feasibility with respect to the constraints
\[
\begin{pmatrix}
\sum_{i\in[n]} \bm{S}^i & \bm{X}^\top\\
\bm{X} & \bm{Y}
\end{pmatrix}\succeq \bm{0},
\qquad
\bm{S}^{i}=\bm{W}_{x,x}^{(i,i)}\succeq \bm{X}_{i,.}\bm{X}_{i,.}^\top,
\]
follows identically to the proof of Theorem~\ref{thm:shorlowrankequiv}.

Second, for any $i\in[n]$, consider the principal submatrix of the lifted PSD constraint in \eqref{eq:fa_shorrelax}
indexed by $(1,\mathrm{vec}(\bm{X}^\top)_{\ell(i,i)}, d_i)$, i.e., $(1,X_{ii},d_i)$. This yields
\[
\begin{pmatrix}
1 & X_{ii} & d_i\\
X_{ii} & (\bm{W}_{x,x})_{\ell(i,i),\ell(i,i)} & (\bm{W}_{x,d})_{\ell(i,i),i}\\
d_i & (\bm{W}_{x,d})_{\ell(i,i),i} & (\bm{W}_{d,d})_{ii}
\end{pmatrix}\succeq \bm{0}.
\]
Using $(\bm{W}_{x,x})_{\ell(i,i),\ell(i,i)}=(\bm{S}^i)_{ii}$ and the definitions of $t_i$ and $v_i$ gives
\[
\begin{pmatrix}
1 & X_{ii} & d_i\\
X_{ii} & S^i_{ii} & v_i\\
d_i & v_i & t_i
\end{pmatrix}\succeq \bm{0},
\]
which is exactly the second set of LMIs in \eqref{eq:fa_sdprelax_compact}.

Finally, the objective values agree because $\mathrm{tr}(\bm{W}_{x,x})=\sum_{i\in[n]}\mathrm{tr}(\bm{W}_{x,x}^{(i,i)})=\sum_{i\in[n]}\mathrm{tr}(\bm{S}^i)$,
$\mathrm{tr}(\bm{W}_{d,d})=\sum_{i\in[n]} t_i$, and
$\langle \bm{P},\bm{W}_{x,d}\rangle=\sum_{i\in[n]} (\bm{W}_{x,d})_{\ell(i,i),i}=\bm{e}^\top \bm{v}$
by the definition of $\bm{P}$.

\textbf{Problem \eqref{eq:fa_shorrelax} is a relaxation of Problem \eqref{eq:fa_sdprelax_compact}:}
Fix a feasible solution $(\bm{X},\bm{d},\bm{Y},\{\bm{S}^i\}_{i\in[n]},\bm{t},\bm{v})$ to \eqref{eq:fa_sdprelax_compact}, with a view to construct a feasible solution to \eqref{eq:fa_shorrelax}.

First, define $\bm{W}_{x,x}\in\mathcal{S}^{n^2}_+$ blockwise via $\bm{W}_{x,x}^{(i,i)} := \bm{S}^i,\
\bm{W}_{x,x}^{(i,k)} := \bm{X}_{i,}\bm{X}_{k,}^\top\ \ (i\neq k).$ Then $\bm{W}_{x,x}\succeq \bm{0}$ by construction and
$\mathrm{tr}(\bm{W}_{x,x})=\sum_{i\in[n]}\mathrm{tr}(\bm{S}^i)$.

Second, let $\bm{\Theta}:=\sum_{i\in[n]}\bm{S}^i$, $\bm{U}:=\bm{\Theta}^\dagger \bm{X}^\top\in\mathbb{R}^{n\times n}$ and set the diagonal blocks
$\bm{W}_{x,y}^{(i,i)} := \bm{S}^i \bm{U}$ and $\bm{W}_{y,y}^{(i,i)} := \bm{U}^\top \bm{S}^i \bm{U}$.
Then, identically to the proof of Theorem~\ref{thm:shorlowrankequiv}, we can complete the remaining (off-diagonal) blocks of
$\bm{W}_{x,y}$ and $\bm{W}_{y,y}$ so that
\[
\begin{pmatrix}
1 & \mathrm{vec}(\bm{X}^\top)^\top & \mathrm{vec}(\widetilde{\bm{Y}}^\top)^\top\\
\mathrm{vec}(\bm{X}^\top) & \bm{W}_{x,x} & \bm{W}_{x,y}\\
\mathrm{vec}(\widetilde{\bm{Y}}^\top) & \bm{W}_{x,y}^\top & \bm{W}_{y,y}
\end{pmatrix}\succeq \bm{0}.
\]

Third, let $\bm{M}_{x,y}$ denote the PSD matrix in the display above, and let
$\bm{g}_0,\{\bm{g}_{p}\}_{p=1}^{n^2},\{\bm{h}_{q}\}_{q=1}^{n^2}$ be vectors whose Gram matrix equals $\bm{M}_{x,y}$
(e.g., take a Cholesky factor of $\bm{M}_{x,y}$). In particular,
$\|\bm{g}_0\|^2=1$, $\bm{g}_0^\top \bm{g}_{\ell(i,i)}=X_{ii}$, and
$\|\bm{g}_{\ell(i,i)}\|^2=(\bm{W}_{x,x})_{\ell(i,i),\ell(i,i)}=S^i_{ii}$. Next, fix $i\in[n]$ and write
\[
\bm{g}_{\ell(i,i)} = X_{ii}\bm{g}_0 + \bm{r}_i,
\qquad \text{where } \bm{r}_i \perp \bm{g}_0,\ \ \|\bm{r}_i\|^2 = S^i_{ii}-X_{ii}^2.
\]
Introduce a new unit vector $\bm{e}_i$ orthogonal to $\mathrm{span}\{\bm{g}_0,\bm{g}_p,\bm{h}_q:\ p,q\in[n^2]\}$,
and define
\[
\alpha_i :=
\begin{cases}
\dfrac{v_i - X_{ii} d_i}{\|\bm{r}_i\|^2}, & \text{if } \|\bm{r}_i\|^2>0,\\[0.6em]
0, & \text{if } \|\bm{r}_i\|^2=0,
\end{cases}
\qquad
\beta_i := \sqrt{t_i - d_i^2 - \alpha_i^2\|\bm{r}_i\|^2}.
\]
The quantity under the square root is nonnegative because the constraint
$\begin{pmatrix}1&X_{ii}&d_i\\X_{ii}&S^i_{ii}&v_i\\d_i&v_i&t_i\end{pmatrix}\succeq \bm{0}$
implies, by Schur complement, that
$\begin{pmatrix}\|\bm{r}_i\|^2 & v_i - X_{ii}d_i\\ v_i - X_{ii}d_i & t_i - d_i^2\end{pmatrix}\succeq \bm{0}$. Now, set $ \bm{q}_i := d_i \bm{g}_0 + \alpha_i \bm{r}_i + \beta_i \bm{e}_i.$
Then, $\bm{g}_0^\top \bm{q}_i=d_i$, $\|\bm{q}_i\|^2=t_i$, and
\[
\bm{g}_{\ell(i,i)}^\top \bm{q}_i
=
X_{ii}d_i + \bm{r}_i^\top (\alpha_i \bm{r}_i)
=
X_{ii}d_i + \alpha_i \|\bm{r}_i\|^2
=
v_i.
\]
Define the remaining lifted blocks by
\[
(\bm{W}_{d,d})_{ij} := \bm{q}_i^\top \bm{q}_j,\qquad
(\bm{W}_{x,d})_{p,i} := \bm{g}_{p}^\top \bm{q}_i,\qquad
(\bm{W}_{y,d})_{q,i} := \bm{h}_{q}^\top \bm{q}_i.
\]
By construction, the full lifted matrix in \eqref{eq:fa_shorrelax} is the Gram matrix of
$\{\bm{g}_0,\bm{g}_p,\bm{h}_q,\bm{q}_i\}$, hence is PSD. Moreover,
$\mathrm{tr}(\bm{W}_{d,d})=\sum_{i\in[n]}\|\bm{q}_i\|^2=\bm{e}^\top \bm{t}$ and
$\langle \bm{P},\bm{W}_{x,d}\rangle=\sum_{i\in[n]} (\bm{W}_{x,d})_{\ell(i,i),i}=\sum_{i\in[n]} \bm{g}_{\ell(i,i)}^\top \bm{q}_i=\bm{e}^\top \bm{v}$.
Therefore, the constructed solution is feasible for \eqref{eq:fa_shorrelax} and has the same objective value as
\eqref{eq:fa_sdprelax_compact}. \Halmos \endproof
\newpage

\begin{bibunit}

\section{Additional Low-Rank Matrix Completion Results Supporting Section \ref{sec:numerics}}
Figure \ref{fig:gw_vary_gamma} compares the quality of different relaxations for low-rank matrix completion by returning the optimality gap achieved, defined as the relative difference between the lower bound (obtained by each relaxation) and our upper bound (obtained by alternating minimization, AM). Figures \ref{fig:gw_vary_gamma_absoluteLBs} and \ref{fig:gw_vary_gamma_absoluteUBs_BM} report the lower and upper bounds separately. Moreover, Figure \ref{fig:gw_vary_gamma_runtimes} compares the same three relaxations in terms of computational time.

Figure \ref{fig:gw_vary_n_2} reports similar results to Figure \ref{fig:gw_vary_n}, except with $\rho=0.2$ rather than $\rho=0.3$. 

\FloatBarrier
\begin{figure}
    \centering
\begin{subfigure}{0.45\textwidth}
        \includegraphics[width=\textwidth]{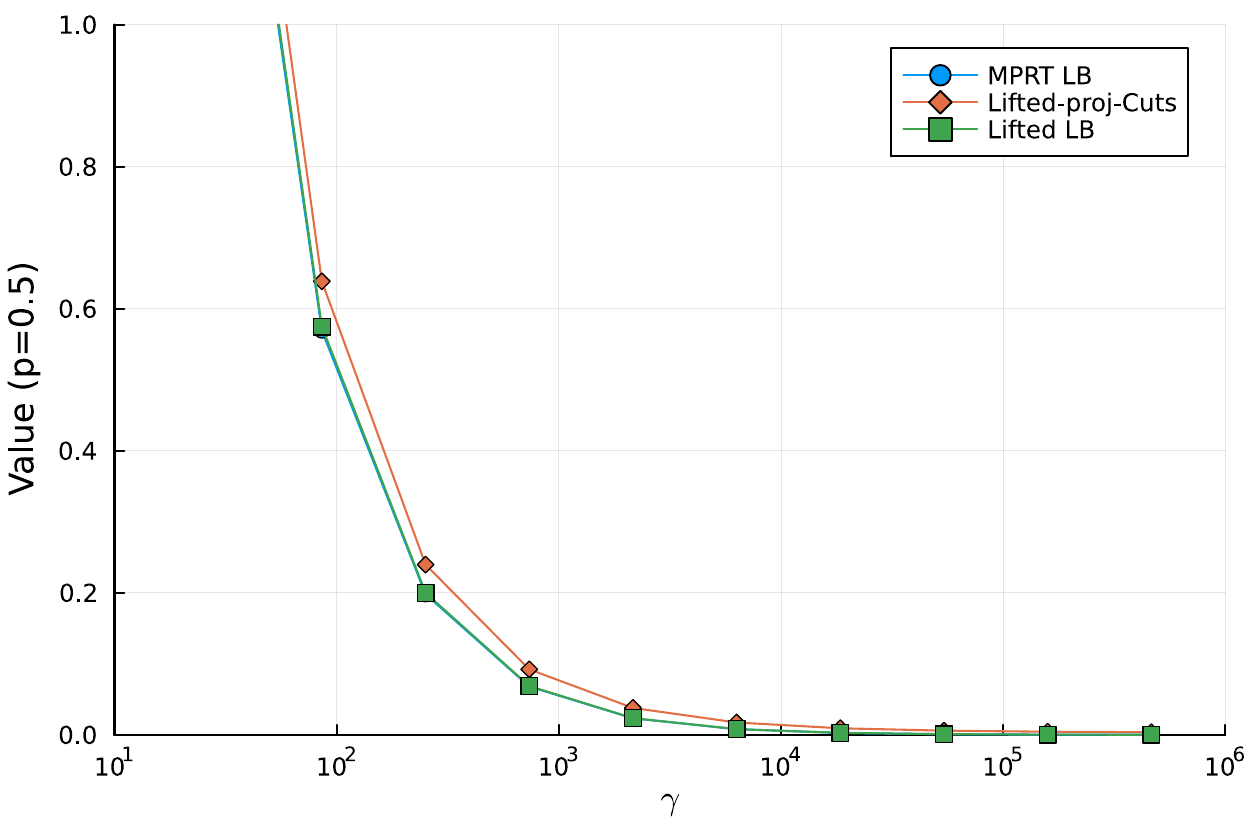}
        \caption{$p=0.5$}
        \label{fig:bounds_0p5_3}
    \end{subfigure}
        \begin{subfigure}{0.45\textwidth}
        \includegraphics[width=\textwidth]{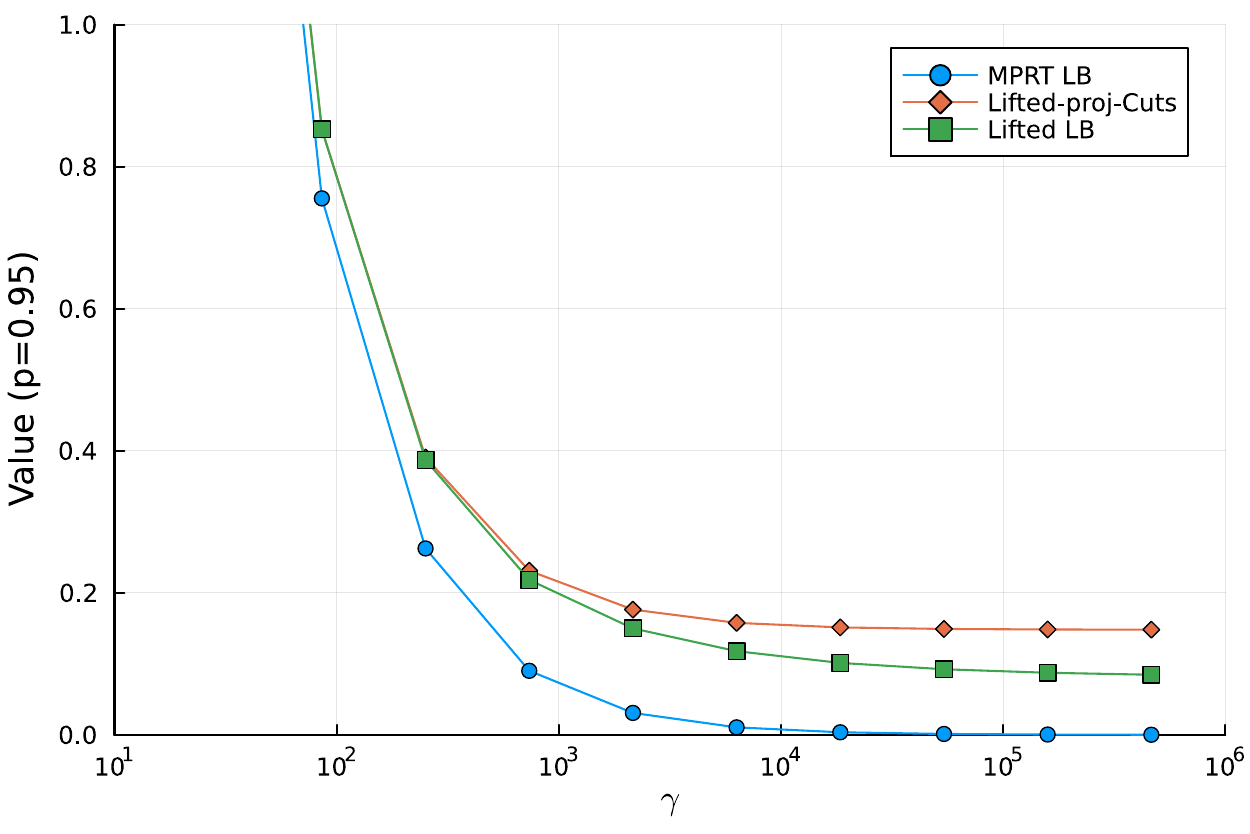}
        \caption{$p=0.95$ }
        \label{fig:bounds_0p95_3}
    \end{subfigure}
    \caption{\color{black}Absolute lower bounds as we vary $\gamma$ for (i) a matrix perspective relaxation (``MPRT''), (ii) our lifted relaxation (``Lifted''), (iii) our lifted relaxation with projection cuts (``Lifted-proj-cuts''), for $p =0.5$ in panel (a) and $p= 0.95$ in panel (b) for $n=8$.}
    \label{fig:gw_vary_gamma_absoluteLBs}
\end{figure}

\begin{figure}
    \centering
\begin{subfigure}{0.45\textwidth}
        \includegraphics[width=\textwidth]{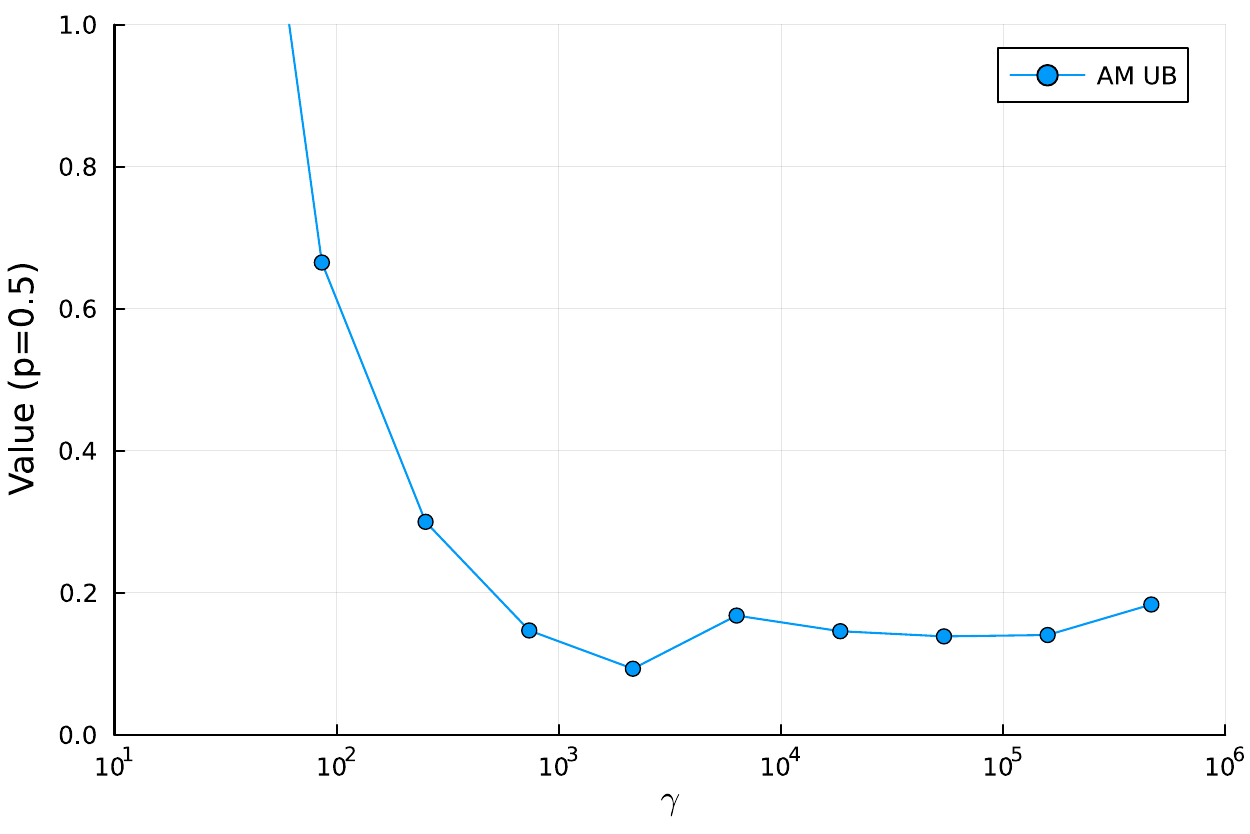}
        \caption{$p=0.5$}
        \label{fig:bounds_0p5_ubs}
    \end{subfigure}
        \begin{subfigure}{0.45\textwidth}
        \includegraphics[width=\textwidth]{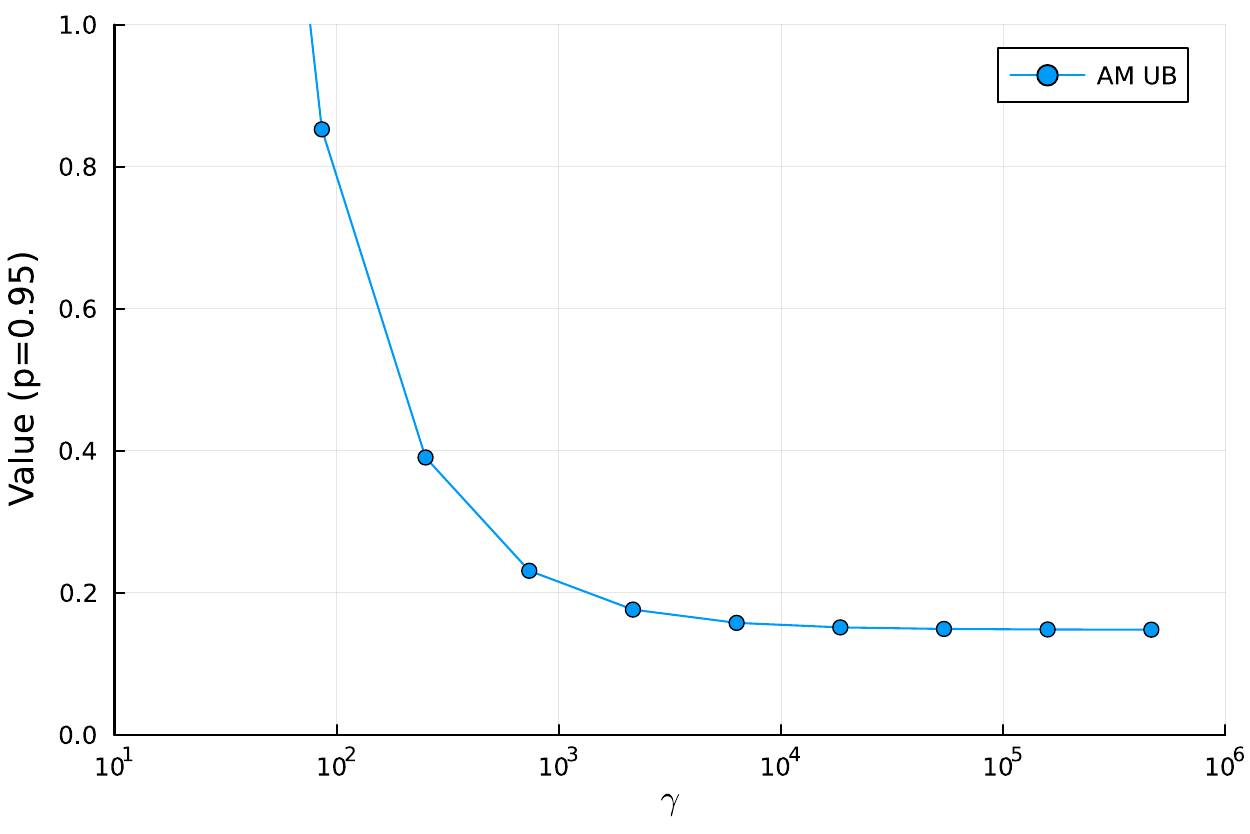}
        \caption{$p=0.95$}
        \label{fig:bounds_0p95_ubs}
    \end{subfigure}
    \caption{Absolute upper bounds as we vary $\gamma$ for alternating minimization initialized at a rank-$k$ SVD of $\mathcal{P}(\bm{A})$ for $p \in \{0.5, 0.95\}$ and $n=8$. }
    \label{fig:gw_vary_gamma_absoluteUBs_BM}
\end{figure}
\FloatBarrier

\begin{figure}
    \centering
\begin{subfigure}{0.45\textwidth}
        \includegraphics[width=\textwidth]{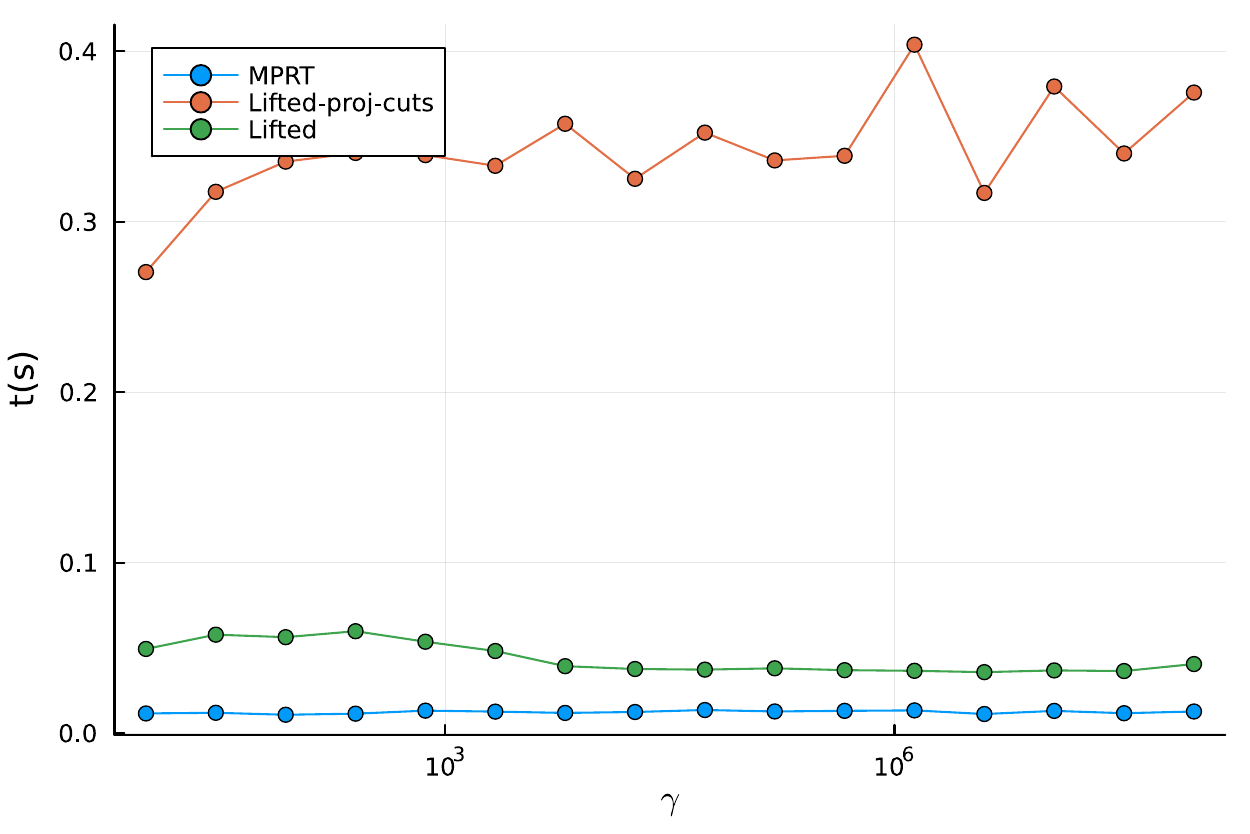}
        \caption{$p=0.5$}
        \label{fig:bounds_0p5_runtimes}
    \end{subfigure}
        \begin{subfigure}{0.45\textwidth}
        \includegraphics[width=\textwidth]{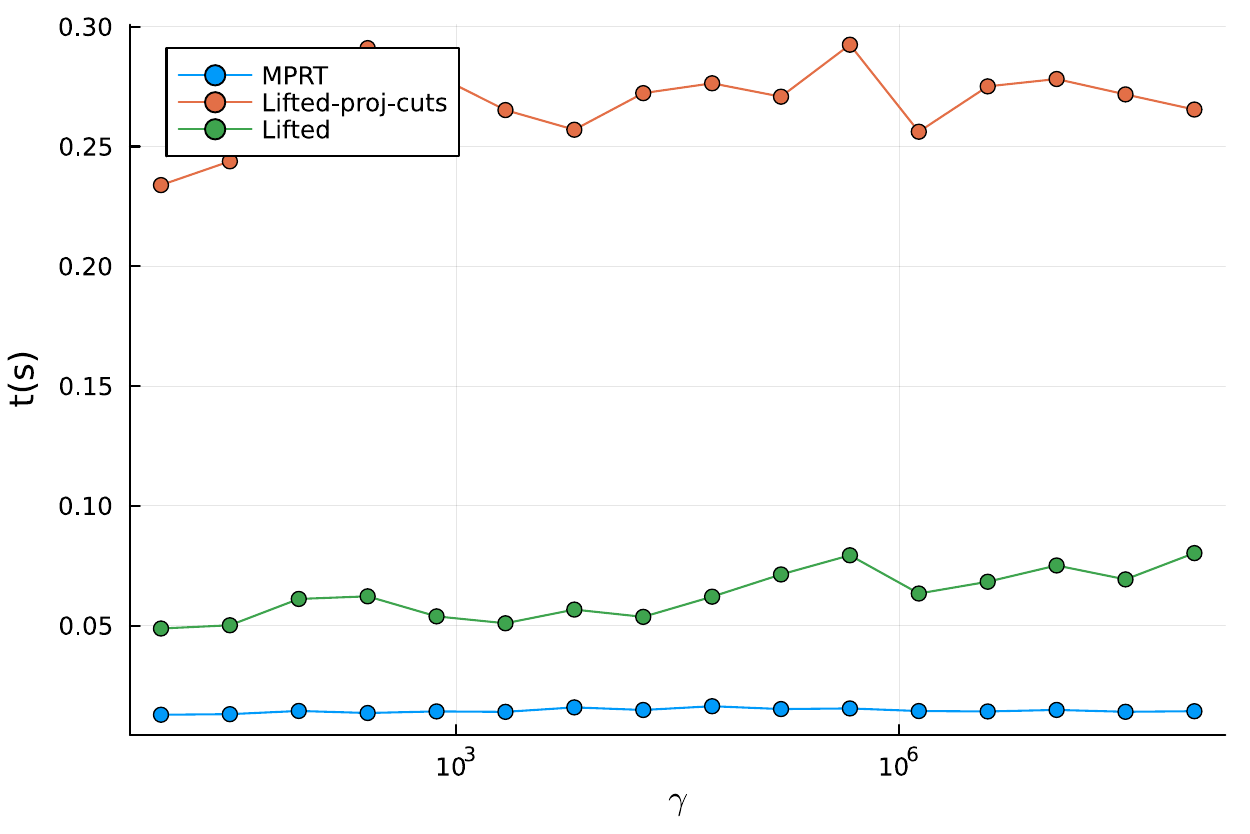}
        \caption{$p=0.95$ }
        \label{fig:bounds_0p95_runtimes}
    \end{subfigure}
    \caption{\color{black}Runtimes as we vary $\gamma$ for (i) a matrix perspective relaxation (``MPRT''), (ii) our lifted relaxation (``Lifted''), (iii) our lifted relaxation with projection cuts (``Lifted-proj-cuts''), for $p =0.5$ in panel (a) and $p= 0.95$ in panel (b) for $n=8$.}
    \label{fig:gw_vary_gamma_runtimes}
\end{figure}

\begin{figure}[h!]
    \centering
\begin{subfigure}{0.45\textwidth}
        \includegraphics[width=\textwidth]{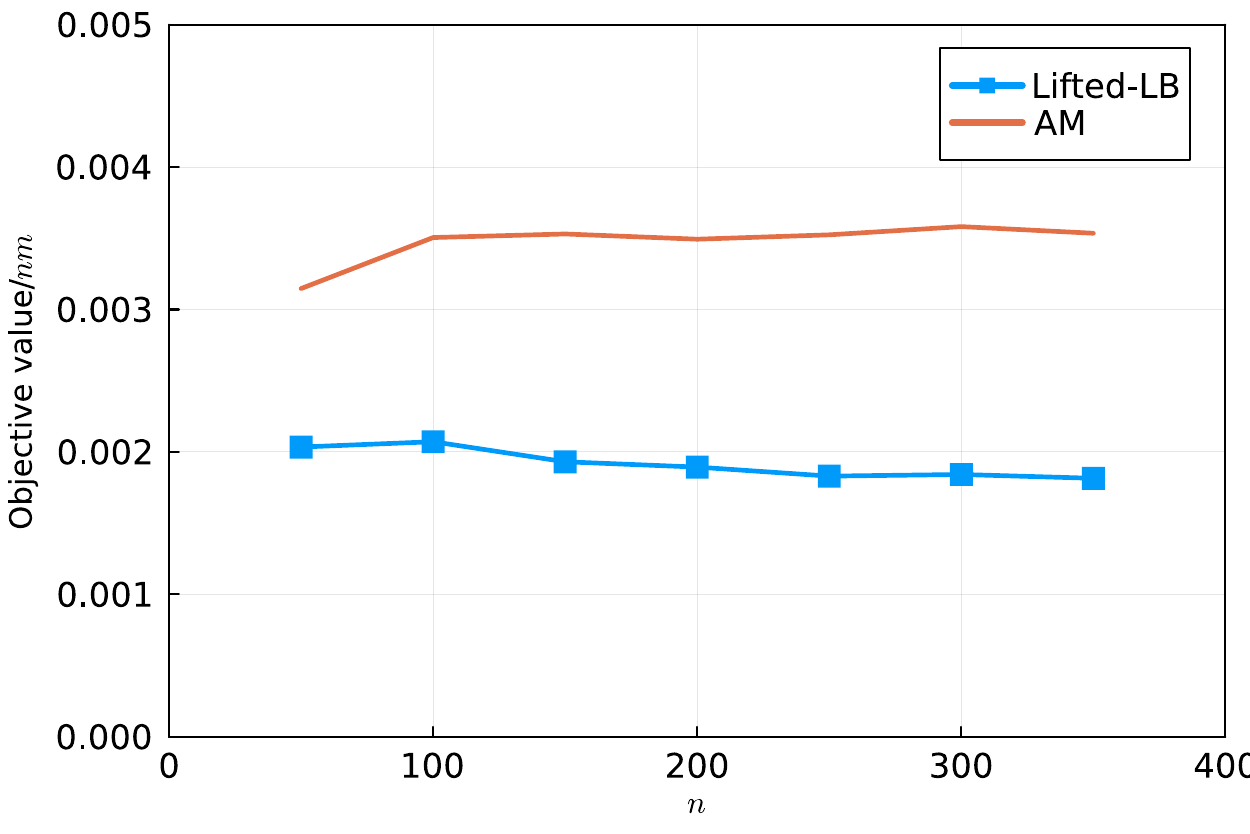}
        \caption{Objective values}
        \label{fig:bounds_0p5_objvals}
    \end{subfigure}
        \begin{subfigure}{0.45\textwidth}
        \includegraphics[width=\textwidth]{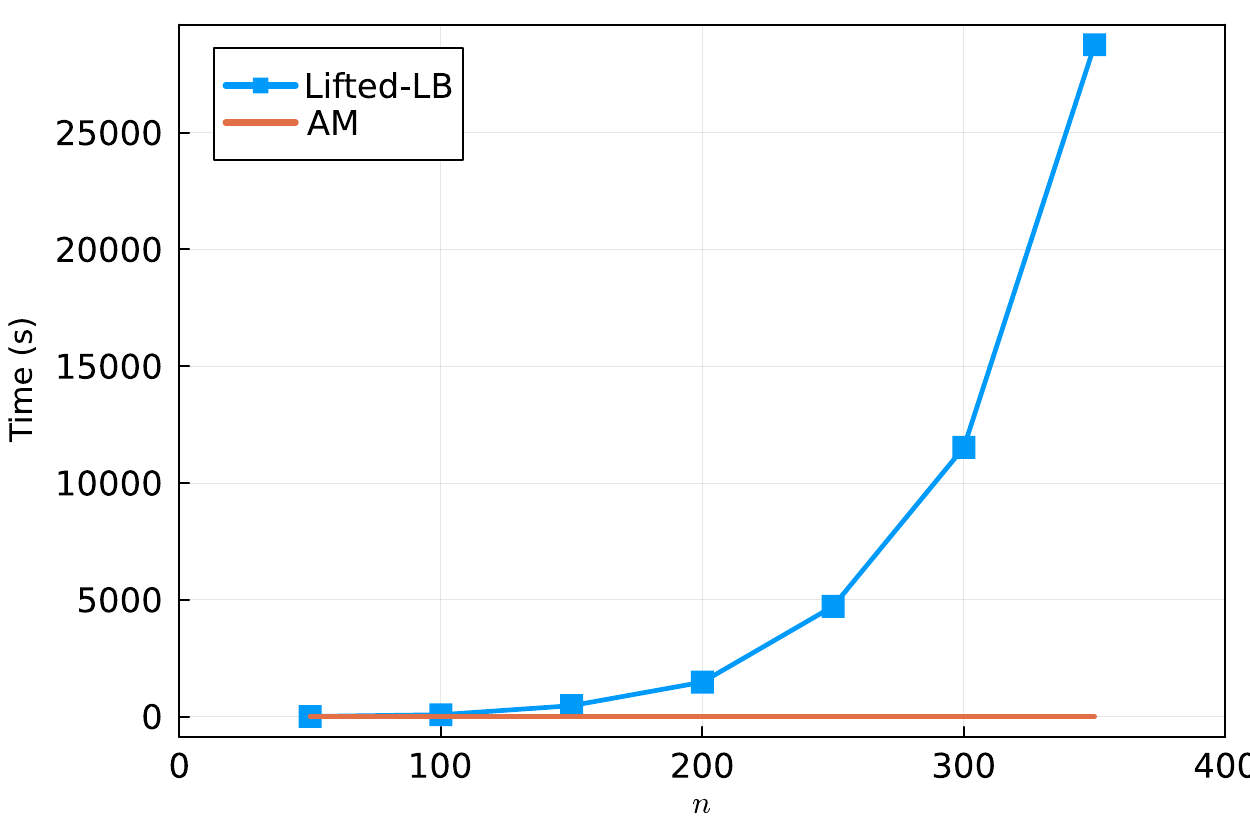}
        \caption{Runtimes}
        \label{fig:bounds_0p95_objvals}
    \end{subfigure}
    \caption{\color{black}Objective value (left panel) and runtime for the compact lifted relaxation (right panel) as we vary $n$ with $m=50$ for our compact lifted relaxation \eqref{eqn:lmrc_further_reduced} with $p=0.4, \rho=0.2$. Results are averaged over 10 replications.}
    \label{fig:gw_vary_n_2}
\end{figure}


\end{bibunit}
\end{document}